\documentclass[12pt,reqno]{amsart}
\usepackage{amsmath,amssymb}
\usepackage[margin=1in]{geometry}
\usepackage[colorlinks=true,linkcolor=blue,citecolor=blue]{hyperref}
\usepackage{xcolor}
\allowdisplaybreaks

\newcommand{\wh}{\widehat}
\newcommand{\wt}{\widetilde}
\renewcommand{\rm}{\mathrm}
\usepackage{mathtools}
\usepackage{enumerate}
\usepackage{bm}
\newcommand{\abs}[1]{\left\lvert#1\right\rvert}
\newcommand{\pa}[1]{\left( #1 \right)}
\newcommand{\rpa}[1]{\left[ #1 \right]}
\newcommand{\br}[1]{\left\lbrace #1\right\rbrace}
\newcommand{\R}{\mathbb{R}}
\newcommand{\N}{\mathbb{N}}

\newcommand{\norm}[1]{\left\lVert#1\right\rVert}
\renewcommand{\phi}{\varphi}
\renewcommand{\epsilon}{\varepsilon}

\newtheorem{theorem}{Theorem}

\newtheorem{lemma}{Lemma}
\newtheorem{proposition}{Proposition}
\newtheorem{corollary}{Corollary}
\newtheorem{remark}{Remark}
\title{Indirect diffusion effect in degenerate reaction--diffusion systems}
\author[A. Einav]{Amit Einav}
\address{Amit Einav \hfill\break
	Institute of Mathematics and Scientific Computing, University of Graz, 
	Heinrichstrasse 36, 8010 Graz, Austria}
\email{amit.einav@uni-graz.at}
\author[J. Morgan]{Jeff Morgan}
\address{Jeff Morgan \hfill\break
	Department of Mathematics, 
	University of Houston, Houston, Texas 77004, USA}
\email{jmorgan@math.uh.edu}
\author[B.Q. Tang]{Bao Quoc Tang}
\address{Bao Quoc Tang \hfill\break
	Institute of Mathematics and Scientific Computing, University of Graz, 
	Heinrichstrasse 36, 8010 Graz, Austria}
\email{quoc.tang@uni-graz.at, baotangquoc@gmail.com}
\begin{document}
\begin{abstract}
In this work we study global well-posedness and large time behaviour for a typical reaction--diffusion system, which include degenerate diffusion, and whose non-linearities arise from chemical reactions. We show that there is an {\it indirect diffusion effect}, i.e. an effective diffusion for the non-diffusive species which is incurred by a combination of diffusion from diffusive species and reversible reactions between the species. Deriving new estimates for such degenerate reaction-diffusion system, we show, by applying the entropy method, that the solution converges exponentially to equilibrium, and provide explicit convergence rates and associated constants.
\end{abstract}

\maketitle

\tableofcontents

\section{Introduction}\label{sec:intro}
\subsection{Reaction-diffusion equations}
In the past few decades, the interest in reaction-diffusion equations and systems has increased, both due to their prevalence in many chemical and physical instances, and their mathematical beauty and intricacy. 
In general, reaction-diffusion equations describe the behaviour of a physical or chemical system where two fundamental processes are competing against each other: the diffusion of matter in the system, and the chemical/physical reactions that govern it.

Many times in the context of reversible chemical processes one encounters the situation where $n$ different chemical compounds, $\br{\mathcal{U}_i}_{i=1,\dots,n}$,  diffuse in a certain medium, while still obeying the reaction    
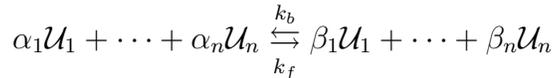
\begin{equation}\nonumber
\alpha_1 \mathcal{U}_1+\dots+\alpha_n\mathcal{U}_n \underset{k_f}{\overset{k_b}{\leftrightarrows}} \beta_1 \mathcal{U}_1+\dots+\beta_n\mathcal{U}_n
\end{equation}
with stoichiometric coefficients $\alpha_i, \beta_i \in \{0\}\cup [1,\infty)$ for $i=1,\ldots, n$, where $k_f>0, k_b>0$ are the reaction rate constants. We would like to mention that it is common in chemistry to only consider stoichiometric coefficients that are natural numbers. By applying {\it mass action kinetics}, the above system is modelled by the following (reaction-diffusion) system for the concentration functions, $u_i(x,t)$, of the compound $\mathcal{U}_i$, which are non-negative functions defined in the domain of the problem, $\Omega \subset \R^N$, and which satisfy a homogeneous Neumann boundary condition on the boundary of that domain:
\begin{equation}\label{eq:general_chemical_rd}
\begin{split}
\partial_t u_i-d_i &\Delta u_i= \pa{\beta_i-\alpha_i}\pa{k_f \prod_{i=1}^n u_i^{\alpha_i}-k_b \prod_{i=1}^n u_i^{\beta_i}},\quad  i=1,\dots,n,\; x\in \Omega,\; t>0
\end{split}
\end{equation}
\begin{equation}\label{eq:general_chemical_neumann}
\nabla u_i \cdot \nu = 0,\quad x\in \partial \Omega,\; t>0,
\end{equation}
where  $\br{d_i}_{i=1,\dots,n}\geq 0$ are the diffusion coefficients, and $\nu(x)$ is the outward normal vector to $\partial \Omega$ at a given point $x\in \partial \Omega$.

\medskip
Clearly, the difficulty in analysing this system increases greatly with $n$ and with the degree of the polynomial that appears on the right hand side of \eqref{eq:general_chemical_rd}. It is not even clear if there is a global solution to this system, although as most of the chemical reactions we view in real life do not end up in an explosion, we have a strong intuition that this is indeed the case (see \cite{PS00} for an "artificial chemical reaction system" in which the solution blows up in $L^\infty$ norm in finite time despite the dissipation of total mass). The issues one must face become more complicated when possible degeneracies in the diffusion of some of the compounds are considered. Such situations have been reported in several works, such as those modelling optical sensors or biosensors (see e.g. \cite{BIK09,JJMS96,MGM98}). However, one can intuitively guess that due to the interactions between the concentrations, governed by the chemical reactions, an effect of \textit{indirect diffusion} will be observed, i.e. a concentration that does not have self diffusion will nonetheless exhibit diffusive like behaviour. Showing that such an intuition is indeed correct, at least as it manifests itself in the question of convergence to equilibrium, is the goal of this presented work. 

\subsection{The setting of the problem}
In our work we will focus on a specific type of reaction-diffusion system of the form \eqref{eq:general_chemical_rd} with $n=4$ and 
$$\alpha_1=\alpha_2=\beta_3=\beta_4=1,\quad \alpha_3=\alpha_4=\beta_1=\beta_2=0,$$
meaning that the compounds $\mathcal{U}_1$ and $\mathcal{U}_2$ completely change to $\mathcal{U}_3$ and $\mathcal{U}_4$, and vice versa:
\begin{equation}\label{binary_reaction}
	\mathcal{U}_1+\mathcal{U}_2 \leftrightarrows \mathcal{U}_3 + \mathcal{U}_4.
\end{equation}
As our focus in this work is on the indirect diffusion effect, we will assume for simplicity that the reaction rate constants are both one, i.e. $k_f = k_b = 1$. This simplification should bear no real impact on our result. In this case, the reaction-diffusion system with {\it full diffusion}, i.e. $d_i > 0$ for $i=1,\ldots, 4$, reads as
\begin{equation}\label{eq:sys_full}
\begin{cases}
\partial_t u_1 - d_1\Delta u_1 &= -u_1u_2 + u_3u_4, \quad x\in\Omega, \; t>0,\\
\partial_t u_2 - d_2\Delta u_2 &= -u_1u_2 + u_3u_4, \quad x\in\Omega, \; t>0,\\
\partial_t u_3 - d_3\Delta u_3 &= +u_1u_2 - u_3u_4, \quad x\in\Omega, \; t>0,\\
\partial_t u_4 - d_4\Delta u_4 &= +u_1u_2 - u_3u_4, \quad x\in\Omega, \; t>0,\\
\end{cases}	
\end{equation}
together with the homogeneous Neumann boundary conditions
\begin{equation}\label{sys_bc}
	\nabla u_i \cdot \nu = 0, \quad \text{ for all } \quad i=1,\ldots,4\quad\text{ and }  x\in\partial\Omega, \; t>0,
\end{equation}
and non-negative initial data
\begin{equation}\label{sys_id}
	u_i(x,0) = u_{i,0}(x) \geq 0 \quad \text{ for all } \quad i=1,\ldots, 4, \; x\in\Omega.
\end{equation}

The key feature of the system we will focus on, is the fact that one of the compounds, say $\mathcal{U}_4$ without loss of generality, will have no self diffusion. That means that, we shall study the well-posedness and large time behaviour of the following system:
\begin{equation}\label{eq:sys}
\begin{cases}
\partial_t u_1 - d_1\Delta u_1 &= -u_1u_2 + u_3u_4, \quad x\in\Omega, \; t>0,\\
\partial_t u_2 - d_2\Delta u_2 &= -u_1u_2 + u_3u_4, \quad x\in\Omega, \; t>0,\\
\partial_t u_3 - d_3\Delta u_3 &= +u_1u_2 - u_3u_4, \quad x\in\Omega, \; t>0,\\
\partial_t u_4 &= +u_1u_2 - u_3u_4, \quad x\in\Omega, \; t>0,\\
\end{cases}	
\end{equation}
subject to homogeneous Neumann boundary conditions \eqref{sys_bc} \emph{only for $i=1,2,3$} and initial data \eqref{sys_id}.  Thanks to the homogeneous Neumann boundary conditions, one can observe that  \eqref{eq:sys_full} (and also \eqref{eq:sys}), possesses, at least formally, the following conservation laws
\begin{equation}\label{eq:mass_conservation}
\int_{\Omega}(u_i(x,t) + u_j(x,t))dx = M_{ij}:= \int_{\Omega}(u_{i,0}(x) + u_{j,0}(x))dx, \quad \forall \; t\geq 0,
\end{equation}
for $i=1,2$, $j=3,4$. This means that no single mass is conserved, but some combinations of them are. This phenomena is not surprising, as the chemical reactions in the systems move one compound to another. We would also remark that exactly three of the four laws in \eqref{eq:mass_conservation} are linearly independent, and these laws allow us to uniquely identify a unique positive equilibrium which balances the reaction \eqref{binary_reaction}. Intuitively speaking we have two competing long term behaviours of \eqref{eq:sys_full}, corresponding to two processes:
\begin{itemize}
\item The diffusion attempts to distribute the compound evenly on $\Omega$, i.e. it tries to push $\br{u_i(x,t)}_{i=1,\dots,4}$ to become constants;
\item The chemical reaction, on the other hand, tries to push the compound to the point where the reaction is balanced, i.e
$$u_1(x,t)u_2(x,t)=u_3(x,t)u_4(x,t).$$
\end{itemize}
Only the combination of these two processes can drive the system \eqref{eq:sys_full} towards chemical equilibrium. Together with the aforementioned conservation laws, we expect that an equilibrium, if it exists, will be of the form $\bm{u}_\infty = (u_{i,\infty})_{i=1,\ldots, 4} \in (0,\infty)^4$, such that:
\begin{equation}\label{equi_sys}
	\begin{cases}
		u_{1,\infty}u_{2,\infty} = u_{3,\infty} u_{4,\infty},\\
		u_{1,\infty} + u_{3,\infty} = |\Omega|^{-1} M_{13},\\
		u_{1,\infty} + u_{4,\infty} = |\Omega|^{-1}M_{14},\\
		u_{2,\infty} + u_{3,\infty} = |\Omega|^{-1}M_{23},
	\end{cases}
\end{equation}
where the values $M_{ij}$ are defined in \eqref{eq:mass_conservation}.
%

\medskip
Despite a large amount of work dealing with large time behaviour of systems with {\it full diffusion} \eqref{eq:sys_full} (or in larger generality), to our knowledge there has been none so far that has dealt with degenerate systems of the form \eqref{eq:sys}. The main aims in this paper are to show global existence of strong solutions to these degenerate reaction-diffusion systems, and to find a quantitative exponential convergence to equilibrium.


\subsection{The current state of the art}
The study of reaction-diffusion systems is a classical topic (see, for instance, \cite{Rot84}), yet it still poses a lot of challenging open questions. One example, concerning the systems of the form \eqref{eq:general_chemical_rd}, is the global well-posedness of the system when the non-linearities of the system have super quadratic growth. We refer the interested reader to the extensive review \cite{Pie10}. 
The case of a quadratic system with {\it full diffusion}, \eqref{eq:sys_full}, has been investigated extensively: In \cite{GV10}, it was shown that when $N\leq 2$, \eqref{eq:sys_full} has a unique global strong solution. Remaining still in two dimensions, this strong solution has later been shown to be uniformly bounded in time in \cite{CDF14,PSY19}. In higher dimensions, i.e. $N\geq 3$, \cite{DFPV07} showed the global existence of a weak solution, but the uniqueness remains open. In the 2014 work \cite{CDF14}, it was shown that conditional global strong solutions for $N\geq 3$ can be obtained if the diffusion coefficients are ``close enough''. Recently, in three parallel studies \cite{CGV19}, \cite{Sou18} and \cite{FMT19}, the system \eqref{eq:sys_full} was shown to have a unique global strong solution in all dimensions.

The investigation of the long time behaviour of these systems has also bore fruit. The trend to equilibrium for \eqref{eq:sys_full} was first studied in one dimension in \cite{DF08} using the so-called \textit{entropy method} (which we will explain shortly). When dealing with higher dimension, the work \cite{CDF14} have shown that under certain conditions (such as the ``closeness'' of diffusion coefficients) one can again recover some relaxation information. Most notably, in the recent work \cite{FMT19}, the authors have shown that the global strong solution, which they have proven to exist, converges exponentially fast to equilibrium in $L^\infty(\Omega)$-norm. We would also like to mention the work \cite{GZ10}, where the author treated \eqref{binary_reaction} with a probabilistic approach (and not entropic methodology). Due to the global existence and bounds found in \cite{FMT19}, the question of large time behaviour of \eqref{eq:sys_full} in all dimensions is now settled. This can be summed in the following theorem.
\begin{theorem}\label{thm:convergence_no_degen}\cite{FMT19}
Assume that $\Omega\subset\mathbb R^N$, with $N\geq 1$, is a bounded domain with boundary, $\partial \Omega$, that is $C^{2+\epsilon}$ for some $\epsilon>0$, with $\Omega$ lying locally only on one side of $\partial\Omega$. Assume in addition that the diffusion coefficients, $d_i$, $i=1,\ldots, 4$ in \eqref{eq:sys_full} are positive. Then for any bounded, non-negative initial data, $\bm{u}_0 = (u_{i,0})_{i=1,\ldots, 4}$, there exists a unique global strong solution to \eqref{eq:sys_full} which converges exponentially to the equilibrium $\bm{u}_\infty$ defined in \eqref{equi_sys}, in $L^\infty(\Omega)$-norm, i.e.
\begin{equation*}
\sum_{i=1}^4 \norm{u_i(t)-u_{i,\infty}}_{L^\infty\pa{\Omega}} \leq K_1 e^{-\lambda_1 t},\quad t\geq 0,
\end{equation*}
where $K_1,\lambda_1 > 0$ are given explicitly, and depend only on $\Omega$, $\br{d_i}_{i=1,\dots, 4}$, the initial masses $M_{ij}$, and $\br{\norm{u_{i,0}}_{L^\infty(\Omega)}}_{i=1,\dots, 4}$.
\end{theorem}

The interested reader can find the most recent advances in the study of the trend to equilibrium in general chemical reaction-diffusion systems (with full diffusion) in \cite{MHM14,FT17,DFT17,FT18,Mie16,MM18}.

\medskip
Needless to say, the analysis of the degenerate system \eqref{eq:sys} is much less satisfactory. The first work studying global existence of \eqref{eq:sys} is \cite{DFPV07} in which a global weak solution was shown in all dimensions. Global strong solutions in one and two dimensions were then shown in \cite{DF15} and \cite{CDF14} respectively. In higher dimensions, global strong solutions were also obtained in \cite{DF15} under an additional condition that the diffusion coefficients are close to each other. 

The convergence to equilibrium for the degenerate system \eqref{eq:sys} is, on the other hand, completely open, as far as we know.
As the large time behaviour of the reversible reaction \eqref{binary_reaction} is the combination of diffusion and chemical reactions, the presence of \textit{diffusion in all species} plays an important role in the commonly used entropy method, with which a convergence to equilibrium is sought. When some diffusion is missing, things become much more involved. In some preliminary works which treated special systems of two species, see, for instance, \cite{DF07} or \cite{FLT18}, the authors utilised a phenomenon, which they called the {\it indirect diffusion effect}, to get convergence to equilibrium using the entropy method. Roughly speaking, the idea behind this phenomenon is that a combination of reversible reactions and diffusion of diffusive species lead to some ``diffusion effect" for the non-diffusive species. While this intuition is natural, showing it quantitatively in the form of a functional inequality, which is the heart of the entropy method, turns out to quite complicated in most cases (for the simple linear case, we refer the reader to \cite{FPT17}). It is worth mentioning, as it will play a role in our investigation as well, that in both \cite{DF07} and \cite{FLT18}, the uniform-in-time bounds of $L^\infty(\Omega)$-norm of solutions are essential. 

In this paper, we push forward the theory of indirect diffusion effects by investigating the degenerate system \eqref{eq:sys}. 
We study the well-posdeness of the system and find suitable estimates for the solution which, while weaker than a uniform $L^\infty\pa{\Omega}$ bound, will suffice for our application of the entropy method. Remarkably, once these estimates are found, and the entropy method is successfully applied to obtain convergence to equilibrium, we are able to use our explicit convergence to revisit our estimates on the solutions to \eqref{eq:sys_full} and improve them up to a uniform $L^\infty(\Omega)$ bound, which will then improve the convergence rate of the entropy method, and yield an exponential convergence. 

\subsection{Main results and key ideas}
The main results of our work concern themselves with the existence, uniqueness and long time behaviour of solutions to our problems. In particular, one can ask oneself what type of solutions one can achieve. As our goal here is not only to deal with the system of PDEs, but also to show the intimate connection to the notions of entropy and indirect diffusion, we elected to consider only classical solutions - that is solutions that are at least $C^{2,1}$ ($C^2$ in space and $C^1$ in time), and solve the differential equations, as well as satisfy the boundary conditions, pointwise (in the classical sense). Many of our proofs can be extended to weaker notions of solutions, yet we leave this discussion out of our present work.\\
In addition, to simplify some of the writing, we decided to employ a couple of notions:
\begin{itemize}
\item We will often refer to conditions \eqref{sys_bc} \textit{only for $i=1,2,3$} as \textit{the compatibility conditions} of our system of equations. Note that without diffusion, the equation for $u_4$ acts as an ODE, which is why we require no boundary conditions. 
\item We will use the notion of \textit{smoothness} to indicate that our initial data and domain itself guarantee the existence of a local classical solution. Following theorems from \cite{Ama85} (as shall be seen in the proofs), this will be valid when the initial data is in $C^2\pa{\overline{\Omega}}$ and that the boundary of $\Omega$, $\partial \Omega$, is $C^{2+\epsilon}$ for some $\epsilon>0$, with $\Omega$ lying locally only on one side of $\partial \Omega$.
\end{itemize}
An important consequence of having classical solutions is that the formal mass conservation equations, given by \eqref{eq:mass_conservation}, are no longer formal but exact, as one can easily differentiate the appropriate integrals. \\
The main results of this paper are the following two theorems.
\begin{theorem}[$N=1,2$]\label{thm:main}
Let $N=1,2$ and let $\Omega$ be a bounded domain of $\R^{N}$ with smooth boundary, $\partial \Omega$. Then for any smooth, non-negative initial data $\br{u_{i,0}(x)}_{i=1,\dots,4}$ that satisfies the compatibility condition, \eqref{eq:sys} has a unique, classical, global non-negative, bounded solution $\br{u_i}_{i=1,\ldots, 4}$ which satisfies the exponential convergence to equilibrium
\begin{equation*}
	\sum_{i=1}^4\norm{u_i(t) - u_{i,\infty}}_{L^\infty(\Omega)} \leq K_2e^{-\lambda_2 t} \quad \text{ for all } \quad t\geq 0,
\end{equation*}
where $K_2>0$ and $\lambda_2>0$ are {\normalfont explicit} constants depending only on $N$, $\Omega$, the initial masses $M_{ij}$, and the diffusion coefficients $d_1, d_2$ and $d_3$. Here $\bm{u}_\infty$ is the unique positive equilibrium determined by the equilibrium system \eqref{equi_sys}.

\end{theorem}

\begin{theorem}[$N\geq 3$]\label{thm:main-3D}
Let $N\geq 3$ and let $\Omega$ be a bounded domain of $\R^N$ with smooth boundary, $\partial \Omega$. There exists an explicit $\delta > 0$ such that if
\begin{equation}\label{quasi-uniform_theorem}
	\frac{|d_i-d_3|}{d_i+d_3} < \delta \quad \text{ for } \quad i = 1 \quad \text{ or } \quad i=2,
\end{equation}
then for any smooth, non-negative initial data $\br{u_{i,0}(x)}_{i=1,\dots,4}$  that satisfies the compatibility condition, \eqref{eq:sys} has a unique, classical, global non-negative, bounded solution $\br{u_i}_{i=1,\ldots, 4}$ which satisfies the exponential convergence to equilibrium
\begin{equation*}
	\sum_{i=1}^4\norm{u_i(t) - u_{i,\infty}}_{L^\infty(\Omega)} \leq K_3e^{-\lambda_3 t} \quad \text{ for all } \quad t\geq 0,
\end{equation*}
where $K_3>0$ and $\lambda_3>0$ are {\normalfont explicit} constants depending only on $\delta$, $N$, $\Omega$, the initial masses $M_{ij}$, and the diffusion coefficients $d_1, d_2$ and $d_3$. Here $\bm{u}_\infty$ is the unique positive equilibrium determined by the equilibrium system \eqref{equi_sys}.
\end{theorem}
\begin{remark}
Condition \eqref{quasi-uniform_theorem} means that either $d_1$ and $d_3$ or $d_2$ and $d_3$ are ``close'' to each other. As was mentioned, the constant $\delta$ can in fact be quantified, and chosen to be 
	\begin{equation*}
		\delta = \frac{1}{C_{\mathrm{mr},p_0'}},
	\end{equation*}
for some $p_0' > 0$ is such that $p_0 = \frac{p_0'}{p_0' - 1} > \frac{N+2}{2}$, and where $C_{\mathrm{mr},p_0'}$ is the optimal constant  in the maximal regularity for a linear parabolic equation (see Lemma \ref{lem_mr}). Note that in fact we have here a family of potential constants, corresponding to $p_0'$ and $C_{\mathrm{mr},p_0'}$, and we only need to find one such constant, for which condition \eqref{quasi-uniform_theorem} is satisfied.
\end{remark}
\begin{remark}\label{rem:mass_bounds}
As was mentioned at the beginning of the subsection, since we have classical solutions, \eqref{eq:mass_conservation} is valid. Tying this together with the non-negativity of the solutions leads to the following set of bounds
\begin{equation}\label{eq:mass_bounds}
\norm{u_i(t)}_{L^1\pa{\Omega}} \leq M_{ij},\quad\quad i=1,2,\;j=3,4,
\end{equation}
We shall use this fact in some of our proofs. 
\end{remark}

Let us briefly sketch the ideas we use to prove Theorems \ref{thm:main} and \ref{thm:main-3D}. 
The global existence of bounded solutions to \eqref{eq:sys} will be shown by duality arguments. More precisely, when $N=1,2$ or when $N\geq 3$ and the required conditions are satisfied, a duality estimate will show that $u_i \in L^p(\Omega\times(0,T))$ for $i=1,2,3$, and some $p > \frac{N+2}{2}$. \\
These estimates are enough to begin a bootstrap process that will end up with $u_i\in L^\infty(\Omega\times(0,T))$ for $i=1,2,3$ and $T>0$ fixed, which in turn will imply that $u_4 \in L^\infty(\Omega\times(0,T))$. 
These results are not surprising, as a similar methodology was used in \cite{CDF14} and \cite{DF15} in the case of one and two dimensions.\\
We would like to emphasise that tools required to achieve the above change drastically between $N=1,2$ and $N\geq 3$ (see for instance Lemma \ref{PW} and Proposition \ref{Global3D}) which is the reason behind the separation of these cases in our main theorems.

\medskip
To consider convergence to equilibrium that is given by the system \eqref{equi_sys}, we need to find a ``distance'' under which the solution converges as time goes to infinity. This aforementioned ``distance'' needs not be a norm or a metric, but a non-negative functional that behaves well under the flow of the equations and will be zero \emph{only} when the solution is at equilibrium. This situation is very common in many physically/chemically motivated systems. In practice what we do is look for a functional for the equation, $H(\bm{u}|\bm{u}_\infty)$, which we call \emph{the entropy} or  \emph{the relative entropy}, such that
\begin{itemize}
\item $H(\bm{u}|\bm{u}_\infty)\geq 0$, and $H(\bm{u}|\bm{u}_\infty)=0$ with $\bm{u}$ satisfying all the conservation laws if and only if $\bm u=\bm u_\infty$.
\item For any solution $\bm u(t)$ to the equation/s, $\frac{d}{dt}H(\bm{u}(t)|\bm{u}_\infty) \leq 0$ (i.e. $H$ is a Lyapunov functional for the system).
\end{itemize}
To find the rate of convergence for the entropy we employ the so-called \emph{entropy method}. Defining the entropy production of $H$ by the functional that formally satisfies 
$$D(\bm u(t))=-\frac{d}{dt}H(\bm u(t)|\bm u_\infty)$$
we forget that $D$ is connected to $H$ via the system, and attempt to find a \emph{functional inequality} that connects $H$ and $D$, of the form
\begin{equation}\label{eede}
D(\bm u) \geq \Psi(H(\bm u|\bm u_\infty))
\end{equation}
for a suitable function $\Psi$. Once such an inequality has been obtained, we recall that under the flow of the equations, $D$ is minus the derivative of $H$, and the above functional inequality becomes a differential inequality which, under some conditions on $\Psi$, yields a quantitative rate of convergence to equilibrium. The case $\Psi(y)=Cy$, which one hopes to get in many physically relevant equations, yields an exponential rate of convergence. An entropy that seems to be appropriate in numerous setting is the so-called \emph{Boltzmann entropy}:
$$H(\bm u|\bm u_\infty)= \int \phi\pa{\frac{\bm u}{\bm u_\infty}}\bm u_\infty dx,$$
where $\phi(x)=x\log x -x +1$. Unsurprisingly, this entropy will be our choice for the entropy of the system \eqref{eq:sys}. More precisely, the relative entropy for \eqref{eq:sys} (and also for \eqref{eq:sys_full}) is
\begin{equation*}
	H(\bm u|\bm u_\infty) = \sum_{i=1}^4\int_{\Omega}\pa{u_i\log\pa{\frac{u_i}{u_{i,\infty}}} - u_i + u_{i,\infty}}dx.
\end{equation*}
A simple calculation (formal in general, but exact when the functions are smooth enough) shows that the entropy production under \eqref{eq:sys} is given by
\begin{equation*}
	D(\bm u) = \sum_{i=1}^3d_i\int_{\Omega}\frac{|\nabla u_i|^2}{u_i}dx + \int_{\Omega}(u_1u_2 - u_3u_4)\log\pa{\frac{u_1u_2}{u_3u_4}}dx
\end{equation*}
where as under \eqref{eq:sys_full} it is given by
\begin{equation*}
	D^{\rm{full}}(\bm u) = \sum_{i=1}^4d_i\int_{\Omega}\frac{|\nabla u_i|^2}{u_i}dx + \int_{\Omega}(u_1u_2 - u_3u_4)\log\pa{\frac{u_1u_2}{u_3u_4}}dx.
\end{equation*}
The first sum in $D(\bm{u})$ (or in $D^{\rm{full}}(\bm{u})$) corresponds to the diffusion of the system, while the second term corresponds to the reversible reaction \ref{binary_reaction}. If all diffusion coefficients are present (i.e. we deal with \eqref{eq:sys_full}), it can be shown that 
\begin{equation*}
	D^{\rm{full}}(\bm u) \geq \lambda H(\bm u|\bm u_\infty),
\end{equation*}
for some $\lambda>0$, which gives an exponential convergence (see e.g. \cite{FT17} or \cite{DFT17}). The degenerate case, however, is more delicate. 

A particular case of only two substances was considered in \cite{DF08} and \cite{FLT18}, where the authors assumed a reversible reaction of the form $\alpha S_1 \leftrightarrows \beta S_2$, $\alpha, \beta \geq 1$, with either $S_1$ or $S_2$ lacking self diffusion. To employ the entropy method, these works utilised the so-called {\it indirect diffusion effect}. Roughly speaking, a combination of diffusion from diffusive species and the reversible reaction leads to some diffusion effect on the non-diffusive species. To successfully show this phenomena, the proofs in \cite{DF08} or \cite{FLT18} essentially used a uniform-in-time $L^\infty(\Omega)$ bound of the solution. Mimicking this strategy, one would like to show for the system considered in this paper that
\begin{equation}\label{IDE}
	D(\bm{u}) \geq \beta\norm{\sqrt{u_4} - \overline{\sqrt{u_4}}}_{L^2(\Omega)}^2,
\end{equation}
assuming that the solution is uniformly bounded in time in $L^\infty(\Omega)$-norm, and $\overline{\sqrt{u_4}}$ is the spatial average of $\sqrt{u_4}$. It is worth noting that when the diffusion of $u_4$ is present, the Poincar\'e-Wirtinger inequality implies that 
$$D^{\rm{full}}(\bm u)  \geq d_4\int_{\Omega}\frac{|\nabla u_4|^2}{u_4}dx = 4d_4\norm{\nabla \sqrt{u_4}}_{L^2(\Omega)}^2 \geq 4d_4C_{\Omega}\norm{\sqrt{u_4} - \overline{\sqrt{u_4}}}_{L^2(\Omega)}^2.$$ Therefore, \eqref{IDE} can be seen as some ``diffusion behaviour" of $u_4$. \\
Unfortunately, as the system \eqref{eq:sys} has degenerate diffusion, it is unclear how to a-priori obtain uniform-in-time $L^\infty(\Omega)$  bounds. Nevertheless, we will be able to show that
\begin{equation}\label{u3Linf}
	\sup_{t\geq 0}\norm{u_3(t)}_{L^\infty(\Omega)} < +\infty
\end{equation}
and
\begin{equation}\label{u124Lq}
	\norm{u_i(t)}_{L^q(\Omega)} \leq C(1+t)^\alpha \quad \text{ for } \quad i=1,2,4,
\end{equation}
with some $C>0$, $0<\alpha<1$ and $q$, which is dimension dependent. 
This will motivate us to (successfully) seek a modified version of \eqref{IDE}, in which the constant $\beta$ is replaced by a function that depends on the $L^q(\Omega)$ norm of the solution, i.e.
\begin{equation}\label{indirect_ineq_modified}
	D(\bm u) \geq \beta\pa{\br{\norm{u_i}_{L^q(\Omega)}}_{i=1,\dots,4}}\norm{\sqrt{u_4} - \overline{\sqrt{u_4}}}_{L^2(\Omega)}^2.
\end{equation}
Finding such inequality, and using the assistance of the Poincar\'e-Wirtinger inequality, we will be able to show that along the flow of \eqref{eq:sys} we have that
\begin{equation*}
	D(\bm{u}(t)) \geq \lambda(t)H(\bm{u}(t)|\bm{u}_\infty)
\end{equation*}
for an explicit non-negative function $\lambda: [0,\infty) \to [0,\infty)$ which satisfies $\lambda(t) \leq C_{\varepsilon}(1+t)^{1-\alpha-\varepsilon}$ for any $0<\varepsilon<1-\alpha$. This will result in a {\it stretched exponential} convergence for the relative entropy, i.e.
\begin{equation*}
	H(\bm{u}(t)|\bm{u}_\infty) \leq H(\bm{u}_0|\bm{u}_\infty)e^{-C_{1,\varepsilon}t^{1-\alpha-\varepsilon}}.
\end{equation*}
While one can stop at this point, the known Csisz\'ar-Kullback-Pinsker inequality (see Lemma \ref{CKP}) implies that the  stretched exponential convergence of the relative entropy passes to the $L^1(\Omega)$ norm of the solution. This fact allows us to interpolate stretched exponential convergence in $L^1(\Omega)$ norm, algebraic growth in $L^\infty(\Omega)$ norm, and the smoothing effect of the heat operator, to obtain stretched exponential convergence in $L^\infty(\Omega)$ norm for $u_1, u_2, u_3$, which will then be transferred to $u_4$ as well. With these newly found bounds we will be able to return to \eqref{indirect_ineq_modified} and obtain an inequality of the form \eqref{IDE}. At this point, similar arguments as in \cite{DF08} and \cite{FLT18} will yield the desired exponential convergence. 
We find this interaction between a-priori bounds and entropic convergence both beautiful and extremely revealing. 
%

\medskip
\noindent{\bf Structure of the paper}. 
As the study of well-posedness and a-priori bounds of \eqref{eq:sys} and the entropy method are, at least at the first iteration, disjoint, we begin with exploring the entropy method and the manifestation of the indirect diffusion phenomena in it in Section \ref{sec:entropy}. 
In Section \ref{sec:bounds} we delve into our reaction-diffusion system and first show the well-posedness of \eqref{eq:sys}, followed by the estimates that are needed to apply the functional inequality from the previous section. Besides their usage in our current study of the long time behaviour of the solutions, we find these estimates to be interesting in their own right. In Section \ref{sec:proof}, we explore the intimate interaction between the entropy and the previously found norm bounds on the solution to significantly improve our initial estimates, and achieve our main theorems. Lastly, in Section \ref{sec:remarks} we discuss what we believe will be a natural path to continue this line of research, and potential extensions to our work.

\medskip
\noindent{\bf Notation}. In this paper, we regularly use the following notation.
\begin{itemize}
\item For any $0\leq \tau < T$,
\begin{equation*}
\Omega_{\tau,T} = \Omega\times(\tau,T), \quad \norm{\cdot}_{L^p(\Omega\times(\tau,T))} = \norm{\cdot}_{L^p(\Omega_{\tau,T})}.
\end{equation*}
When $\tau = 0$, we write $\Omega_T$ instead of $\Omega_{0,T}$ for simplicity.
\item For any function $f: \Omega \to \mathbb R$ we denote the spatial average of $f$ by
\begin{equation*}
\overline{f} = \frac{1}{|\Omega|}\int_{\Omega}f(x)dx.
\end{equation*}
\item $C_T$ {\it always} denotes a general constant which depends {\it at most algebraically} on the time horizon $T>0$.
\item We often write $C(\alpha, \beta, \gamma, \ldots)$ for a generic constant that depends on the arguments $\alpha, \beta, \gamma$, etc, but {\it not on} the time horizon $T>0$.
\end{itemize}

\section{Indirect diffusion effect and the entropy method}\label{sec:entropy}
This section is divided into two parts: In subsection \ref{FI} we prove the functional inequality that quantifies the indirect diffusion effect that we have mentioned before, and in subsection \ref{sec:convergence} we show how one can utilise this inequality to obtain convergence to equilibrium to our system \eqref{eq:sys}, under the assumption on the growth in time of relevant norms.
\subsection{The entropic inequalities}\label{FI}
In this subsection we will focus solely on the functional inequality between the entropy (defined in Section \ref{sec:intro})
\begin{equation}\nonumber
H\pa{\bm{u}|\bm{u}_{\infty}}=\sum_{i=1}^4 \int_{\Omega}\left(u_i\log\pa{\frac{u_i}{u_{i,\infty}}} - u_i + u_{i,\infty}\right)dx,
\end{equation}
and its production 
\begin{equation}\label{eq:entropy_production}
D\pa{\bm{u}}=\sum_{i=1}^3 \int_{\Omega} d_i \frac{\abs{\nabla u_i}^2}{u_i}dx+ \int_{\Omega}\pa{u_1u_2-u_3u_4}\log \pa{\frac{u_1u_2}{u_3u_4}}dx.
\end{equation}
The first term in this expression is the known \emph{Fisher information}, which is commonly associated with the entropy method and the Boltzmann entropy, while the second term is connected to the chemical reaction and is always non-negative. This second term will be invaluable to get the ``indirect diffusion'' effect, we have discussed in Section \ref{sec:intro}.

Our main theorem for this section is the following:
\begin{theorem}\label{thm:entropy_method}
	Let $\{u_i\}_{i=1,\ldots, 4}$ be non-negative functions on $\Omega \subset \mathbb R^N$. Denoting by
	\begin{equation}\label{eq:masses}
		\int_{\Omega}(u_{i}(x) + u_{j}(x))dx = M_{ij}=\abs{\Omega}\pa{u_{i,\infty}+u_{j,\infty}} \quad\text{ for } \quad i\in \{1, 2\} \quad \text{and} \quad j\in \{3,4\}.
	\end{equation}
 we find that
	\begin{equation}\label{eq:entropy_method}
		\frac{H(\bm{u}|\bm{u}_\infty)}{D(\bm{u})} \leq K_1\pa{1+\max_{i=1,\ldots, 4}\{\log(\|u_i\|_{L^\infty(\Omega)} + 1)\}}\left(1+\max_{i=1,4}\|u_i\|_{L^q(\Omega)}\right)
	\end{equation}
	where
	\begin{equation*}
		q = \begin{cases}
			\frac N2 &\text{ when } N\geq 3,\\
			1+\gamma \text{ for an abitrary fixed } \gamma > 0 &\text{ when } N = 2,\\
			1 &\text{ when } N = 1,
		\end{cases}
	\end{equation*}
and the constant $K_1$ depends only on the domain $\Omega$, the dimension $N$, the diffusion coefficients $d_1, d_2, d_3$, the initial masses $M_{13}, M_{23}, M_{14}$, and in the case where $N=2$, it also depends on $\gamma>0$. In \eqref{eq:entropy_method}, $D(\bm{u})$ is understood to be $+\infty$ when $\sqrt{u_i}$ is not in $H^1(\Omega)$ for some $i=1,2,3$.
\end{theorem}
The proof of this Theorem is involved and is therefore divided to several lemmas. In what follows we will denote by $U_i=\sqrt{u_i}$, $i=1,\ldots, 4$, and notice that
$$u_i\in L^1\pa{\Omega}\;\Rightarrow\; U_i\in L^2\pa{\Omega},$$
$$\nabla U_i = \frac{\nabla u_i}{2\sqrt{u_i}}\quad\text{and as such}\quad \frac{\abs{\nabla u_i}^2}{u_i}=4\abs{\nabla U_i}^2.$$

\medskip
One ingredient to prove Theorem \ref{thm:entropy_method} is the following Poincar\'e-Wirtinger inequality:
\begin{lemma}(Poincar\'e-Wirtinger-type inequality)\label{PW}
	There exists a constant $0<C_{\Omega,q}<\infty$ depending on the domain $\Omega$ and $q$ such that
	\begin{equation}\label{eq:PW}
		\|\nabla f\|_{L^2(\Omega)} \geq C_{\Omega,q}\|f - \overline{f}\|_{L^q(\Omega)}
	\end{equation}
	for all $f\in H^1(\Omega)$, where
	\begin{equation*}
		q = \begin{cases}
			\frac{2N}{N-2} &\text{ for } N\geq 3,\\
			\in [1,\infty) \text{ arbitrary } &\text{ for } N = 2,\\
			\infty &\text{ for } N = 1.
		\end{cases}
	\end{equation*}
	Naturally, when $q = \infty$, the constant $C_{\Omega,q}$ depends only on $\Omega$.
\end{lemma}
\begin{proof}
If $N\geq 3$ inequality \eqref{eq:PW} is the result of the Sobolev embedding theorem, $H^1\pa{\Omega}\subset L^{\frac{2N}{N-2}}\pa{\Omega}$, and the Poincar\'e-Wirtinger inequality
$$\norm{f - \overline{f}}_{L^2\pa{\Omega}} \leq \mathcal{C}_{\Omega,2}\norm{\nabla f}_{L^2\pa{\Omega}}$$
(see, for instance, \cite{Evans}). Indeed, one find that
$$\|f - \overline{f}\|_{L^{\frac{2N}{N-2}}(\Omega)}\leq S_{\Omega,N}\|f - \overline{f}\|_{H^1(\Omega)}$$
$$= S_{\Omega,N}\pa{\|f - \overline{f}\|_{L^2(\Omega)}+\|\nabla f\|_{L^2(\Omega)}}\leq  S_{\Omega,N}\pa{1+\mathcal{C}_{\Omega,2}}\|\nabla f\|_{L^2(\Omega)}.$$
The case $N=2$ (and $p=2$) is a bit more delicate, as it is the critical case. However, as $\Omega$ is bounded we have that $H^1\pa{\Omega}\subset W^{1,\eta}\pa{\Omega}$ for any $1\leq \eta<2$, and as such the choice $\eta=\frac{2q}{q+2}$, together with the Sobolev embedding and the Poincar\'e inequality, yield the desired result.\\ 
Lastly, the case $N=1$ follows from the Sobolev embedding $H^1\pa{\Omega}\subset L^\infty\pa{\Omega}$, and the Poincar\'e inequality again.
\end{proof}

\begin{lemma}\label{lem:first_D_estimation}
Let $\br{u_i}_{i=1,\dots,4}$ be non-negative functions on $\Omega$. Then 
\begin{equation}\label{eq:first_D_estimation}
D\pa{\bm{u}} \geq  \wt{D}(\bm{U}):= 4\pa{\sum_{i=1}^3 d_i\norm{\nabla U_i}_{L^2\pa{\Omega}}^2+\norm{U_1U_2-U_3U_4}_{L^2\pa{\Omega}}^2}.
\end{equation}
Note that $\wt{D}(\bm U)$, and consequently $D(\bm u)$, is understood to be $+\infty$ when $U_i$ is not in $H^1(\Omega)$ for some $i=1,2,3$.
\end{lemma}
\begin{proof}
Using the inequality $\pa{x-y}\log\pa{\frac{x}{y}}\geq 4 \pa{\sqrt{x}-\sqrt{y}}^2 $ we see that
$$\int_{\Omega}\pa{u_1(x)u_2(x)-u_3(x)u_4(x)}\log \pa{\frac{u_1(x)u_2(x)}{u_3(x)u_4(x)}}dx$$
$$ \geq 4\int_{\Omega}\pa{\sqrt{u_1(x)u_2(x)}-\sqrt{u_3(x)u_4(x)}}^2dx=4\norm{U_1U_2-U_3U_4}_{L^2\pa{\Omega}}^2. $$
The connection between the $L^2$ norm of $\nabla U_i$ and the Fisher information of $u_i$ completes the proof.
\end{proof}
The next series of Lemmas express the ``indirect diffusion'' phenomenon we keep mentioning.
\begin{lemma}[$N\geq 3$]\label{lem:second_D_estimation_N>2}
Let $\br{u_i}_{i=1,\dots,4}$ be bounded non-negative functions on a bounded domain $\Omega\subset \R^N$, with $N\geq 3$. Then there exists an explicit $K_2>0$, depending only on $\Omega$, $\br{d_i}_{i=1,2,3}$, $M_{13},M_{23},M_{14}$, such that
\begin{equation}\label{eq:second_D_estimation_N>2}
\begin{gathered}
\wt{D}\pa{\bm{U}}
\geq \frac{K_2}{1 +\max_{i=1,4}\pa{\norm{u_i}_{L^{\frac N2}\pa{\Omega}}}}\norm{U_4-\overline{U_4}}^2_{L^2\pa{\Omega}}
\end{gathered}
\end{equation}
where $\wt D(\bm{U})$ is defined in \eqref{eq:first_D_estimation}. Note that $\wt{D}(\bm{U})$ is understood to be $+\infty$ if $U_i$ is not $H^1(\Omega)$ for some $i=1,2,3$.
\end{lemma}
\begin{proof}
By applying Poincar\'e-Wirtinger inequality from Lemma \ref{PW}, we have
\begin{equation}\label{first_estimate}
\begin{aligned}
\wt D(\bm{U})&\geq 4\sum_{i=1}^3 d_i\norm{\nabla U_i}_{L^2\pa{\Omega}}^2+4\norm{U_1U_2-U_3U_4}_{L^2(\Omega)}^2\\
&\geq 4C_{\Omega}^2\sum_{i=1}^3d_i\|U_i - \overline{U_i}\|_{L^{\frac{2N}{N-2}}(\Omega)}^2 + 4\|U_1U_2 - U_3U_4\|_{L^2(\Omega)}^2\\
&\geq 4\min_{i=1,2,3}\{C_\Omega^2d_i,1\}\widehat{D}(\bm{U})
\end{aligned}
\end{equation}
where
\begin{equation*}
	\widehat{D}(\bm{U}) = \sum_{i=1}^3\|U_i - \overline{U_i}\|_{L^{\frac{2N}{N-2}}(\Omega)}^2 + \|U_1U_2 - U_3U_4\|_{L^2(\Omega)}^2.
\end{equation*}
We consider the zero average functions
$$\delta_i(x)=U_i(x)-\overline{U_i},\quad\quad i=1,\dots,4,$$
with whom we can rewrite
$$
\widehat{D}(\bm{U}) = \sum_{i=1}^3\|\delta_i\|_{L^{\frac{2N}{N-2}}(\Omega)}^2 + \|U_1U_2 - U_3U_4\|_{L^2(\Omega)}^2.
$$
From this point onwards we estimate the above expressions with respect to the ``smallness'' of $\delta_i$, and analyse possible cases. We start with a few simple estimates on $\br{\delta_i}_{i=1,\dots,4}$. From \eqref{eq:masses}, and the non-negative of all $\br{u_i}_{i=1,\dots,4}$, we see that for $M=\max\br{M_{1,3},M_{2,3},M_{1,4}}$
\begin{equation*}
\|U_i\|_{L^2(\Omega)}^2 = \int_{\Omega}u_idx \leq M \quad \text{ for all } \quad i=1,\ldots, 4,
\end{equation*}
Since $$\overline{U^2_i}=\frac{\|U_i\|_{L^2(\Omega)}^2}{\abs{\Omega}}$$ we see that for all $i=1,\ldots, 4$,
$$\overline{U^2_i}\leq \frac{M}{\abs{\Omega}},$$
and 
$$\norm{\delta_i}^2_{L^2\pa{\Omega}}=|\Omega|(\overline{U_i^2}-\overline{U_i}^2) \leq |\Omega|\overline{U_i^2} \leq M,$$
Similarly, as
$$|\Omega|\overline{U_i}^2=|\Omega|\overline{U_i^2}-\norm{\delta_i}^2_{L^2\pa{\Omega}} \leq M$$
we find that
$$\overline{U_i}^2 \leq  \frac{M}{\abs{\Omega}}.$$
%
Next, for a given $\epsilon>0$, to be chosen shortly (see \eqref{epsilon}) we consider the following two cases:

\medskip
\noindent\textbf{Case I: some $\delta_i$ is large.} If 
$$\max_{i=1,2,3}\norm{\delta_i}^2_{L^2\pa{\Omega}} \geq \epsilon,$$
then
\begin{equation}\label{second_estimate}
\begin{aligned}
\widehat{D}\pa{\bm{U}} &\geq \max_{i=1,2,3}\norm{\delta_i}^2_{L^{\frac{2N}{N-2}}\pa{\Omega}} \geq |\Omega|^{-2/N}\max_{i=1,2,3}\|\delta_i\|_{L^2(\Omega)}^2\\
&\geq |\Omega|^{-2/N}\epsilon \geq \frac{\epsilon \abs{\Omega}^{-2/N}}{M} \norm{\delta_4}^2_{L^2\pa{\Omega}}=\frac{ \epsilon \abs{\Omega}^{-2/N}}{M} \norm{U_4-\overline{U_4}}^2_{L^2\pa{\Omega}}. 
\end{aligned}
\end{equation}

\medskip
\noindent\textbf{Case II: all $\delta_i$ are small.} If  
$$\max_{i=1,2,3}\norm{\delta_i}^2_{L^2\pa{\Omega}} \leq \epsilon.$$
then
\begin{equation}\label{eq:crucial_estimate_D_part_I}
\begin{aligned}
\norm{U_1U_2-U_3U_4}_{L^2\pa{\Omega}}^2&=\norm{\pa{\overline{U_1}+\delta_1}\pa{\overline{U_2}+\delta_2}-\pa{\overline{U_3}+\delta_3}U_4}_{L^2\pa{\Omega}}^2\\
 &\geq \frac{\norm{\overline{U_1}\;\overline{U_2}-\overline{U_3}U_4}_{L^2\pa{\Omega}}^2}{2}-\norm{\delta_1\overline{U_2}+\delta_2\overline{U_1}+\delta_1\delta_2 -\delta_3 U_4}^2_{L^2\pa{\Omega}}
\end{aligned}
\end{equation}
where we have used the inequality $\pa{a-b}^2 \geq \frac{a^2}{2}-b^2$. Using the facts that $\pa{a+b+c+d}^2 \leq 4\pa{a^2+b^2+c^2+d^2}$ and $\overline{U_i}^2\leq M/\abs{\Omega}$, we find that
\begin{equation}\label{eq:crucial_estimate_D_part_II}
\begin{aligned}
&\norm{\delta_1\overline{U_2}+\delta_2\overline{U_1}+\delta_1\delta_2 -\delta_3 U_4}^2_{L^2\pa{\Omega}}\\
&\leq \frac{4M}{\abs{\Omega}} \pa{\norm{\delta_1}^2_{L^2\pa{\Omega}}+\norm{\delta_2}^2_{L^2\pa{\Omega}}}
+4\int_{\Omega}|\delta_1\delta_2|^2dx
+4\norm{U_4^2\delta_3^2}_{L^1\pa{\Omega}}\\
&\leq 4M\abs{\Omega}^{2/N-1} \pa{\norm{\delta_1}^2_{L^{\frac{2N}{N-2}}\pa{\Omega}}+\norm{\delta_2}^2_{L^{\frac{2N}{N-2}}\pa{\Omega}}}
+4\int_{\Omega}|\delta_1(x)\delta_2(x)|^2dx
+4\norm{U_4^2\delta_3^2}_{L^1\pa{\Omega}}.
\end{aligned}
\end{equation}
Using H\"older's inequality we see that
\begin{equation}\label{change_1}
\begin{aligned}
\int_{\Omega}|\delta_1(x)\delta_2(x)|^2dx &= \int_{\Omega}|U_1(x) - \overline{U_1}|^2|\delta_2(x)|^2dx \leq\|U_1 - \overline{U_1}\|_{L^{N}(\Omega)}^2 \|\delta_2\|_{L^{\frac{2N}{N-2}}(\Omega)}^2\\
&\leq 2\left(\|U_1\|_{L^N(\Omega)}^2 + |\Omega|^{\frac{2}{N}}\overline{U_1}^2\right) \|\delta_2\|_{L^{\frac{2N}{N-2}}(\Omega)}^2\\
&\leq 2\left(\|u_1\|_{L^{\frac N2}(\Omega)} + M\abs{\Omega}^{\frac{2}{N}-1}\right)\|\delta_2\|_{L^{\frac{2N}{N-2}}(\Omega)}^2.
\end{aligned}
\end{equation}
Similarly,
\begin{equation}
\label{change_2}
\norm{U_4^2\delta_3^2}_{L^1\pa{\Omega}} \leq \norm{U_4}^2_{L^{N}\pa{\Omega}} \norm{\delta_3}^2_{L^{\frac{2N}{N-2}}\pa{\Omega}}  =  \norm{u_4}_{L^{\frac{N}{2}}\pa{\Omega}}\norm{\delta_3}^2_{L^{\frac{2N}{N-2}}\pa{\Omega}}.
\end{equation}
Combining \eqref{eq:crucial_estimate_D_part_I}, \eqref{eq:crucial_estimate_D_part_II}, \eqref{change_1} and \eqref{change_2}, we see that
\begin{equation}\label{eq:crucial_estimate_D_part_III}
\begin{aligned}
&\norm{U_1U_2-U_3U_4}_{L^2\pa{\Omega}}^2\geq \frac{\norm{\overline{U_1}\;\overline{U_2}-\overline{U_3}U_4}_{L^2\pa{\Omega}}^2}{2}-4M\abs{\Omega}^{2/N-1} \norm{\delta_1}^2_{L^{\frac{2N}{N-2}}\pa{\Omega}}\\
&-4\pa{3M\abs{\Omega}^{2/N-1} +2\norm{u_1}_{L^{\frac{N}{2}}\pa{\Omega}}}\norm{\delta_2}^2_{L^{\frac{2N}{N-2}}\pa{\Omega}}-4\norm{u_4}_{L^{\frac{N}{2}}\pa{\Omega}}\norm{\delta_3}^2_{L^{\frac{2N}{N-2}}\pa{\Omega}}
\end{aligned}
\end{equation}
With this at hand, we see that for any $\eta\in (0,1]$ we have that
\begin{equation*}
\begin{aligned}
\widehat{D}\pa{\bm{U}}&\geq \sum_{i=1}^3 \norm{\delta_i}_{L^{\frac{2N}{N-2}}\pa{\Omega}}^2+\eta\norm{U_1U_2-U_3U_4}_{L^2\pa{\Omega}}^2\\
&\geq \left(1-4\eta M |\Omega|^{\frac 2N - 1}\right)\|\delta_1\|_{L^{\frac{2N}{N-2}}(\Omega)}^2\\
&\quad + \left(1 - 4\eta\left(3M|\Omega|^{\frac 2N - 1}+2\|u_1\|_{L^{\frac N2}(\Omega)} \right)\right)\|\delta_2\|_{L^{\frac{2N}{N-2}}(\Omega)}^2\\
&\quad + \left(1-4\eta\|u_4\|_{L^{\frac N2}(\Omega)} \right)\|\delta_3\|_{L^{\frac{2N}{N-2}}(\Omega)}^2 + \frac{\eta}{2}\|\overline{U_1}\,\overline{U_2} - \overline{U_3}U_4\|_{L^2(\Omega)}^2.
\end{aligned}
\end{equation*}
Thus, choosing 
\begin{equation}\label{eta}
\eta=\frac{1}{4\left(3M|\Omega|^{\frac 2N -1} + 2\max_{i=1,4}\|u_i\|_{L^{\frac N2}(\Omega)} \right)+1}<1
\end{equation}
yields
\begin{equation}\label{third_estimate}
\widehat{D}\pa{\bm{U}}\geq  \frac{\eta}{2}\norm{\overline{U_1}\;\overline{U_2}-\overline{U_3}U_4}_{L^2\pa{\Omega}}^2.
\end{equation}
We now estimate the right hand side of \eqref{third_estimate}. This involves two sub-cases, regarding the size of  $\overline{U_3}$.

\medskip
\textbf{Sub-case A: $\overline{U_3}$ is large}. Assume that $\overline{U_3} \geq \sqrt{\epsilon}$. Then, motivated by the connection $u_{4,\infty} = \frac{u_{1,\infty}u_{2,\infty}}{u_{3,\infty}}$, we write 
$$U_4(x)=\frac{\overline{U_1}\;\overline{U_2}}{\overline{U_3}}\pa{1+\xi(x)} \quad \text{ for } \quad x\in\Omega.$$
As such 
$$\norm{\overline{U_1}\;\overline{U_2}-\overline{U_3}U_4}_{L^2\pa{\Omega}}^2=\pa{\overline{U_1}\;\overline{U_2}}^2 \norm{\xi}^2_{L^2\pa{\Omega}}=\pa{\overline{U_1}\;\overline{U_2}}^2  \overline{\xi^2}|\Omega|.$$
On the other hand
$$\norm{U_4-\overline{U_4}}^2_{L^2\pa{\Omega}}= \frac{\pa{\overline{U_1}\;\overline{U_2}}^2}{\overline{U_3}^2}\norm{\xi-\overline{\xi}}^2_{L^2\pa{\Omega}}=\frac{\pa{\overline{U_1}\;\overline{U_2}}^2}{\overline{U_3}^2} \pa{\overline{\xi^2}-\overline{\xi}^2}|\Omega|.$$
Thus 
\begin{equation}\label{fourth_estimate}
\norm{\overline{U_1}\;\overline{U_2}-\overline{U_3}U_4}_{L^2\pa{\Omega}}^2 \geq \epsilon\frac{\pa{\overline{U_1}\;\overline{U_2}}^2}{\overline{U_3}^2} |\Omega|  \overline{\xi^2} \geq \epsilon\norm{U_4-\overline{U_4}}^2_{L^2\pa{\Omega}}.
\end{equation}

\medskip
\textbf{Sub-case B: $\overline{U_3}$ is small.} Assume now that $\overline{U_3} \leq \sqrt{\epsilon}$. In this case since
$$\norm{\delta_3}^2_{L^2\pa{\Omega}}= |\Omega|(\overline{U_3^2}-\overline{U_3}^2),$$
and we are in Case II, we see that $\overline{U_3^2} \leq \overline{U_3}^2 + |\Omega|^{-1}\|\delta_3\|_{L^2(\Omega)}^2 \leq \epsilon(1+|\Omega|^{-1})$. Thus
$$\overline{U_1}^2 = \overline{U_1^2} -|\Omega|^{-1} \norm{\delta_1}^2_{L^2\pa{\Omega}}=\overline{U_1^2}+\overline{U_3^2} -\overline{U_3^2}- |\Omega|^{-1}\norm{\delta_1}^2_{L^2\pa{\Omega}} \geq \frac{M_{13}}{|\Omega|} - \epsilon(1+2|\Omega|^{-1}). $$
Similarly, $\overline{U_2}^2\geq \frac{M_{23}}{|\Omega|} - \epsilon(1+2|\Omega|^{-1})$. With that, and using again the facts that $\pa{a-b}^2 \geq \frac{a^2}{2}-b^2$ and that $\norm{U_i}_{L^2\pa{\Omega}}^2=\abs{\Omega}\overline{U_i^2}$, we have
\begin{equation*}
\begin{aligned}
\norm{\overline{U_1}\;\overline{U_2}-\overline{U_3}U_4}_{L^2\pa{\Omega}}^2 &\geq \frac{\abs{\Omega}}{2}\pa{\overline{U_1}^2\overline{U_2}^2 - 2 \overline{U_3}^2 \overline{U_4^2}}\\
&\geq \frac{\abs{\Omega}}{2}\pa{\pa{\frac{M_{13}}{\abs{\Omega}} - \epsilon(1+2|\Omega|^{-1})}\pa{\frac{M_{23}}{\abs{\Omega}} - \epsilon(1+2|\Omega|^{-1})}-2M\epsilon}\\
&\geq \frac{\abs{\Omega}}{2}\pa{\pa{\frac{\min\{M_{13}, M_{23} \}}{\abs{\Omega}} - \epsilon(1+2|\Omega|^{-1})}^2-2M\epsilon}.
\end{aligned}
\end{equation*}
Therefore, if we choose
\begin{equation}\label{epsilon}
	\varepsilon = \min\left\{\frac{\min\{M_{13},M_{23} \}^2}{16 M|\Omega|^2}, \frac{\min\{M_{13},M_{23} \}}{2\pa{2+|\Omega|}}
	\right\}
\end{equation}
then
\begin{equation*}
\norm{\overline{U_1}\,\overline{U_2} - \overline{U_3}U_4}^2_{L^2(\Omega)} \geq  \frac{|\Omega|}{2}\frac{\min\{M_{13},M_{23} \}^2}{8|\Omega|^2} = \frac{\min\{M_{13},M_{23} \}^2}{16|\Omega|}.
\end{equation*}
Much like in Case I, this \emph{uniform bound} and the fact that $\norm{U_4 - \overline{U_4}}_{L^2(\Omega)}^2 \leq M$ imply that
\begin{equation}\label{fifth_estimate}
\norm{\overline{U_1}\;\overline{U_2}-\overline{U_3}U_4}_{L^2\pa{\Omega}}^2 \geq \frac{|\Omega|\min\{M_{13},M_{23} \}^2}{16M}\norm{U_4-\overline{U_4}}^2_{L^2\pa{\Omega}}.
\end{equation}
It is important to note that this is \emph{the only case/sub-case} where an explicit $\epsilon$ was needed. All other cases are valid for \emph{any} $\epsilon>0$. 

\medskip
Combining \eqref{first_estimate}, \eqref{second_estimate}, \eqref{third_estimate}, \eqref{fourth_estimate} and \eqref{fifth_estimate} we have
\begin{equation*}
\begin{gathered}
	\wt D(\bm{U}) \geq 4\min_{i=1,2,3}\{C_{\Omega}^2d_i, 1\}\min\left\{\frac{\varepsilon|\Omega|^{-\frac 2N}}{M}, \frac{\varepsilon \eta}{2},\frac{\eta}{2}\frac{|\Omega|\min\{M_{13}, M_{23} \}^2 }{16M}  \right\}\norm{U_4 - \overline{U_4}}_{L^2(\Omega)}^2\\
	\geq 4\min_{i=1,2,3}\{C_{\Omega}^2d_i,1\}\eta \min\left\{\frac{\varepsilon|\Omega|^{-\frac 2N}}{M}, \frac{\varepsilon }{2}, \frac{|\Omega|\min\{M_{13}, M_{23} \}^2 }{32M}  \right\}\norm{U_4 - \overline{U_4}}_{L^2(\Omega)}^2
	\end{gathered}
\end{equation*}
where $\varepsilon$ is in \eqref{epsilon} and $\eta$ is in \eqref{eta}. Therefore, the desired estimate \eqref{eq:second_D_estimation_N>2} follows immediately with a suitable explicit constant $K_2>0$.

\end{proof}
\begin{lemma}[$N=2$]\label{lem:second_D_estimation_N=2}
Let $\br{u_i}_{i=1,\dots,4}$ be non-negative functions on a bounded domain $\Omega\subset \R^2$. Then for any $\gamma>0$, there exists an explicit $K_{3,\gamma}>0$, depending only on $\gamma$, $\Omega$, $\br{d_i}_{i=1,2,3}$, $M$ (defined in Lemma \ref{lem:second_D_estimation_N>2}) such that
\begin{equation}\label{eq:second_D_estimation_N=2}
\begin{gathered}
\wt{D}(\bm{U})
\geq \frac{K_{3,\gamma}}{1 +\max_{i=1,4}\pa{\norm{u_i}_{L^{1+\gamma}\pa{\Omega}}}}\norm{U_4-\overline{U_4}}^2_{L^2\pa{\Omega}}
\end{gathered}
\end{equation}
where $\wt{D}(\bm{U})$ is defined in \eqref{eq:first_D_estimation}. Note that $\wt{D}(\bm{U})$ is understood to be $+\infty$ if $U_i$ is not $H^1(\Omega)$ for some $i=1,2,3$.
\end{lemma}
\begin{proof}
The proof of this lemma is almost identical to that of Lemma \ref{lem:second_D_estimation_N>2}, with a few small changes, which we will clarify here.

\medskip
Using the two dimensional version of the Poincar\'e-Wirtinger inequality from Lemma \ref{PW}, we conclude that for any $\gamma>0$ and $i=1,2, 3$,
\begin{equation*}
	\norm{\nabla U_i}_{L^2(\Omega)}^2 \geq C_{\Omega,\gamma}^2\|U_i - \overline{U_i}\|_{L^{\frac{2(1+\gamma)}{\gamma}}}^2,
\end{equation*}
which implies that $\widehat{D}(\bm{U})$ from the previous proof can be changed to 
$$
\widehat{D}(\bm{U}) = \sum_{i=1}^3\|\delta_i\|_{L^{\frac{2\pa{1+\gamma}}{\gamma}}(\Omega)}^2 + \|U_1U_2 - U_3U_4\|_{L^2(\Omega)}^2,
$$
and we obtain
\begin{equation}\nonumber
\begin{aligned}
\wt D(\bm{U})\geq 4\min_{i=1,2,3}\{C_{\Omega,\gamma}^2d_i,1\}\widehat{D}(\bm{U}).
\end{aligned}
\end{equation}
Next, since
\begin{equation*}
\norm{\delta_i}_{L^{\frac{2(1+\gamma)}{\gamma}}(\Omega)}^2  \geq |\Omega|^{-\frac{1}{1+\gamma}}\norm{\delta_i}_{L^2(\Omega)}^2.
\end{equation*}
we see that estimate \eqref{second_estimate} in Case I, i.e.
$\max_{i=1,2,3}\norm{\delta_i}^2_{L^2\pa{\Omega}} \geq \epsilon,$
changes to 
\begin{equation}\nonumber
\begin{aligned}
\widehat{D}\pa{\bm{U}} \geq \frac{ \epsilon \abs{\Omega}^{\frac{\gamma}{1+\gamma}}}{M} \norm{U_4-\overline{U_4}}^2_{L^2\pa{\Omega}}. 
\end{aligned}
\end{equation}
Turning our attention to \eqref{eq:crucial_estimate_D_part_II}, \eqref{change_1} and \eqref{change_2}, we see that 
\begin{equation}\nonumber
\begin{aligned}
&\norm{\delta_1\overline{U_2}}^2_{L^2\pa{\Omega}} \leq \overline{U_2}^2\abs{\Omega}^{\frac{1}{1+\gamma}}\norm{\delta_1}_{L^{\frac{2(1+\gamma)}{\gamma}}(\Omega)}^2 \leq \abs{\Omega}^{-\frac{\gamma}{1+\gamma}}M\norm{\delta_1}_{L^{\frac{2(1+\gamma)}{\gamma}}(\Omega)}^2
\end{aligned}
\end{equation}
and similarly
\begin{equation}\nonumber
\begin{aligned}
&\norm{\delta_2\overline{U_1}}^2_{L^2\pa{\Omega}}\leq \abs{\Omega}^{-\frac{\gamma}{1+\gamma}}M\norm{\delta_2}_{L^{\frac{2(1+\gamma)}{\gamma}}(\Omega)}^2.
\end{aligned}
\end{equation}
Also,
\begin{equation*}
\begin{aligned}
	\int_{\Omega}|\delta_1(x)\delta_2(x)|^2dx &\leq \norm{U_1 - \overline{U_1}}_{L^{2(1+\gamma)}(\Omega)}^2\norm{\delta_2}_{L^{\frac{2(1+\gamma)}{\gamma}}(\Omega)}^2\\
	&\leq 2\left(\norm{U_1}_{L^{2(1+\gamma)}(\Omega)}^2 + \abs{\Omega}^{\frac{1}{1+\gamma}}\overline{U_1}^2\right)\norm{\delta_2}_{L^{\frac{2(1+\gamma)}{\gamma}}(\Omega)}^2\\
	&\leq 2\left(\norm{u_1}_{L^{1+\gamma}(\Omega)}+\abs{\Omega}^{-\frac{\gamma}{1+\gamma}}M \right)\norm{\delta_2}_{L^{\frac{2(1+\gamma)}{\gamma}}(\Omega)}^2
\end{aligned}
\end{equation*}
and
\begin{equation*}
\norm{U_4^2\delta_3^2}_{L^1(\Omega)} \leq \norm{U_4}_{L^{2(1+\gamma)}(\Omega)}^2\norm{\delta_3}_{L^{\frac{2(1+\gamma)}{\gamma}}(\Omega)}^2 = \|u_4\|_{L^{1+\gamma}(\Omega)}\norm{\delta_3}_{L^{\frac{2(1+\gamma)}{\gamma}}(\Omega)}^2.
\end{equation*}
From this point onwards the proof is identical, with $\frac{N}{2}$ replaced by $1+\gamma$, and as such we omit it.
\end{proof}
\begin{lemma}[$N=1$]\label{N=1}
Let $\br{u_i}_{i=1,\dots,4}$ be non-negative functions on a bounded, open interval $\Omega\subset \R$. Then there exists an explicit $K_{4}>0$, depending only on $\Omega$, $\br{d_i}_{i=1,2,3}$, $M$ (defined in Lemma \ref{lem:second_D_estimation_N>2}) such that
\begin{equation}\label{IDE-1D}
\wt{D}(\bm{U})
\geq K_4\norm{U_4-\overline{U_4}}^2_{L^2\pa{\Omega}}
\end{equation}
where $\wt{D}(\bm{U})$ is defined in \eqref{eq:first_D_estimation}. 
Note that $\wt{D}(\bm{U})$ is understood to be $+\infty$ if $U_i$ is not $H^1(\Omega)$ for some $i=1,2,3$.
\end{lemma}
\begin{proof}
Much like the Proof of Lemma \ref{lem:second_D_estimation_N=2}, the proof of this lemma is based on the proof of Lemma \ref{lem:second_D_estimation_N>2}, with appropriate modification due to improved embedding of the Poincar\'e-Wirtinger inequality. In the one dimensional setting, this equates to the embedding  $H^1(\Omega)\hookrightarrow L^\infty(\Omega)$.\\
Again, we only show the required changes in our appropriate estimates. As was mentioned, the  Poincar\'e-Wirtinger inequality in $N=1$ reads as
\begin{equation*}
\norm{\frac{d}{dx}U_i}_{L^2(\Omega)}^2 \geq C_{\Omega}^2\norm{U_i - \overline{U_i}}_{L^\infty(\Omega)}^2,
\end{equation*}
giving us the modified production
$$
\widehat{D}(\bm{U}) = \sum_{i=1}^3\|\delta_i\|_{L^{\infty}(\Omega)}^2 + \|U_1U_2 - U_3U_4\|_{L^2(\Omega)}^2.
$$
Since
\begin{equation*}
	\norm{\delta_i}_{L^\infty(\Omega)}^2 \geq |\Omega|^{-1}\norm{\delta_i}_{L^2(\Omega)}^2,
\end{equation*}
Case I is again secured. The estimates \eqref{eq:crucial_estimate_D_part_II}, \eqref{change_1} and \eqref{change_2} become
\begin{equation}\nonumber
\begin{aligned}
&\norm{\delta_1\overline{U_2}}^2_{L^2\pa{\Omega}} \leq M\norm{\delta_1}_{L^{\infty}(\Omega)}^2 \\
&\norm{\delta_2\overline{U_1}}^2_{L^2\pa{\Omega}} \leq M\norm{\delta_2}_{L^{\infty}(\Omega)}^2
\end{aligned}
\end{equation}
and 
\begin{equation*}
\begin{gathered}
	\int_{\Omega}|\delta_1(x)\delta_2(x)|^2dx \leq \norm{U_1 - \overline{U_1}}_{L^2(\Omega)}^2\norm{\delta_2}_{L^\infty(\Omega)}^2\\
	 \leq 2\left(\|u_1\|_{L^1(\Omega)} + \abs{\Omega}\overline{U_1}^2 \right)\norm{\delta_2}_{L^\infty(\Omega)}^2 \leq 4M\norm{\delta_2}_{L^\infty(\Omega)}^2.
	\end{gathered}
\end{equation*}
Lastly
\begin{equation*}
	\norm{U_4^2\delta_3^2}_{L^1(\Omega)} \leq \norm{U_4}_{L^2(\Omega)}^2\norm{\delta_3}_{L^\infty(\Omega)}^2 \leq M\norm{\delta_3}_{L^\infty(\Omega)}^2.
\end{equation*}
From this point onwards the proof is again exactly as in Lemma \ref{lem:second_D_estimation_N>2}, though there is \emph{no longer dependency in $\br{u_i}_{i=1,\dots,4}$}, which is why $K_4$ is a fixed constant that depends only on the masses $M_{1,3},M_{2,4}$ and $M_{1,4}$.
\end{proof}
We now have the tools to prove our desired functional inequality. 
\begin{proof}[Proof of Theorem \ref{thm:entropy_method}]
We start by considering the function 
$$\Phi(x,y)=\frac{x\log\pa{\frac{x}{y}}-x+y}{\pa{\sqrt{x}-\sqrt{y}}^2}=\frac{\phi\pa{\frac{x}{y}}}{\pa{\sqrt{\frac{x}{y}}-1}^2},$$
where $\varphi(z) = z\log z - z + 1$. 
It is simple to see that $\Phi$ is continuous on any subset of $\pa{\R_+\cup\br{0}}\times\R_+$. Moreover, as 
$$\lim_{s\rightarrow\infty}\frac{\phi(s)}{s\log s}=1,\quad \lim_{s\rightarrow 0^+}\phi(s)=1$$
we see that we can find a uniform constant, $C>0$, such that $\Phi(x,y) \leq C\max\pa{1,\log \pa{\frac{x}{y}}}$.\\
Thus, we have that 
\begin{equation}\label{eq:first_entropy_estimation}
\begin{aligned}
H\pa{\bm{u}|\bm{u_\infty}}&=\sum_{i=1}^4\int_{\Omega}\phi\pa{\frac{u_i(x)}{u_{i,\infty}}}u_{i,\infty}dx\\
&=\sum_{i=1}^4\int_{\Omega}\Phi\pa{u_i(x),u_{i,\infty}} \pa{U_i(x)-U_{i,\infty}}^2 dx\\
&\leq C\max_{i=1,\dots,4}\pa{1,\log \pa{{\norm{u_i}_{L^\infty\pa{\Omega}}+ 1}}+\abs{\log u_{i,\infty}}}\sum_{i=1}^4 \norm{U_i-U_{i,\infty}}^2_{L^2\pa{\Omega}}.
\end{aligned}
\end{equation}
Using $(a+b+c)^2 \leq 3(a^2+b^2+c^2)$, we have
\begin{align*}
\norm{U_i-U_{i,\infty}}^2_{L^2\pa{\Omega}} &\leq 3\pa{\norm{U_i-\overline{U_i}}^2_{L^2\pa{\Omega}}+\norm{\overline{U_i} -\sqrt{\overline{U_i^2}} }_{L^2(\Omega)}^2 + \norm{\sqrt{\overline{U_i}^2}-U_{i,\infty}}^2_{L^2\pa{\Omega}}}\\
&\leq 6\pa{\norm{U_i-\overline{U_i}}^2_{L^2\pa{\Omega}} + |\Omega|\abs{\sqrt{\overline{U_i^2}}-U_{i,\infty}}^2}
\end{align*}
where at the last step we used the fact that 
\begin{equation*}
\abs{\sqrt{\overline{U_i^2}}-\overline{U_i}}^2=\overline{U_i^2}+\overline{U_i}^2 - 2 \sqrt{\overline{U_i^2}}\overline{U_i} \underset{\overline{U_i} \leq \sqrt{\overline{U_i^2}}}{\leq} \overline{U_i^2}-\overline{U_i}^2 =\frac{\norm{U_i-\overline{U_i}}^2_{L^2\pa{\Omega}}}{\abs{\Omega}}
\end{equation*}
Therefore,
\begin{equation}\label{estimate_H}
	H(\bm{u}|\bm{u}_\infty) \leq C\pa{1+\max_{i=1,\ldots, 4}\log\pa{\norm{u_i}_{L^\infty(\Omega)}+1}}\!\pa{\sum_{i=1}^4\norm{U_i-\overline{U_i}}^2_{L^2\pa{\Omega}} + \sum_{i=1}^4\left|\sqrt{\overline{U_i^2}}-U_{i,\infty}\right|^2}.
\end{equation}
From Lemmas \ref{lem:second_D_estimation_N>2}, \ref{lem:second_D_estimation_N=2} and \ref{N=1}, we know that we can find a constant $K_5$ such that
\begin{equation*}
	D(\bm{u}) \geq \frac{K_5}{1+\max_{i=1,\ldots, 4}\pa{\norm{u_i}_{L^q(\Omega)}} }\norm{U_4 - \overline{U_4}}_{L^2(\Omega)}^2
\end{equation*}
where $q = \frac N2$ when $N\geq 3$, $q = 1+\gamma$ for an arbitrary $\gamma > 0$ when $N=2$, and
\begin{equation*}
	D(\bm{u}) \geq K_5\norm{U_4 - \overline{U_4}}_{L^2(\Omega)}^2,
\end{equation*}
when $N=1$. Using the above with the definition of $D(\bm{u})$, Lemma \ref{lem:first_D_estimation} and the Poincar\'e inequality
$$\norm{U_i - \overline{U_i}}_{L^2(\Omega)}^2 \leq C_{\Omega}\norm{\nabla U_i}_{L^2}^2,$$
we see that we can find an appropriate constant $K_6$ that depends only on $\Omega$ and the diffusion coefficients, such that 
\begin{equation}\label{first_D}
\begin{gathered}
	\frac 12 D(\bm{u})= \frac 14 D(\bm{u})+\frac 14 D(\bm{u}) \geq C_{\Omega,\br{d_i}_{i=1,\dots,3}} \sum_{i=1}^3\norm{U_i - \overline{U_i}}_{L^2(\Omega)}^2\\
	+\frac{K_5}{4\pa{1+\max_{i=1,\ldots, 4}\pa{\norm{u_i}_{L^q(\Omega)}}} }\norm{U_4 - \overline{U_4}}_{L^2(\Omega)}^2 \\
	\geq \frac{K_6}{1+\max_{i=1,\ldots, 4}\pa{\norm{u_i}_{L^q(\Omega)}} }\sum_{i=1}^4\norm{U_i - \overline{U_i}}_{L^2(\Omega)}^2,
	\end{gathered}
\end{equation}
and 
\begin{equation}\label{second_D}
	\frac 12 D(\bm{u}) \geq \frac{K_6}{1+\max_{i=1,\ldots, 4}\pa{\norm{u_i}_{L^q(\Omega)}}}\rpa{\sum_{i=1}^4\norm{U_i - \overline{U_i}}_{L^2(\Omega)}^2 + \norm{U_1U_2 - U_3U_4}_{L^2(\Omega)}^2}.
\end{equation}
Using the estimate
\begin{equation}\label{third_D}
	\sum_{i=1}^4\norm{U_i - \overline{U_i}}_{L^2(\Omega)}^2 + \norm{U_1U_2 - U_3U_4}_{L^2(\Omega)}^2 \geq K_7\sum_{i=1}^4\abs{\sqrt{\overline{U_i^2}} - U_{i,\infty}}^2,
\end{equation}
which was shown in greater generality in \cite[estimate (61) in Section 4]{FLT19}, we can now combine \eqref{estimate_H}, \eqref{first_D}, \eqref{second_D} and \eqref{third_D} to conclude \eqref{eq:entropy_method} and finish the proof of Theorem \ref{thm:entropy_method}.
\end{proof}
\subsection{Conditional convergence to equilibrium}\label{sec:convergence}
The previous subsection was dedicated to finding the functional inequality that governs the entropy method, applied to our system of equations. In this short subsection we will show how we can use this inequality to obtain quantitative decay to zero of the relative entropy, under the assumption that the solutions to \eqref{eq:sys} have a certain growth for appropriately relevant norms, appearing in Theorem \ref{thm:entropy_method}. Showing such estimates will be the focus of our next Section, where we will study the global well-posedness for \eqref{eq:sys}. We have elected to leave this short subsection relatively general, to both emphasise the generality of the method and to allow for future progress in bounds of our system \eqref{eq:sys} to have an immediate consequence in the question of convergence to equilibrium. 

\medskip
Our only theorem for this subsection is
\begin{theorem}\label{thm:convergence_general}
Let $\Omega$ be a bounded domain of $\R^N$, $N\geq 1$, with smooth boundary, $\partial \Omega$, and let $\br{u_i(t)}_{i=1,\dots,4}$ be a classical solution for the system of equations \eqref{eq:sys}. Assume that there exist constants $\mu \geq 0$, $0\leq \alpha<1$, and $C_{\mu,\alpha}>0$ such that
\begin{equation}\label{eq:convergence_bound_condition}
\sup_{t\geq 0}\pa{\max_{i=1,\dots,4}(1+t)^{-\mu }\norm{u_i(t)}_{L^\infty\pa{\Omega}} + \max_{i=1,4}(1+t)^{-\alpha}\norm{u_i(t)}_{L^q(\Omega)}} \leq C_{\mu,\alpha},
\end{equation}
where
\begin{equation*}
	q = \begin{cases}
		\frac N2 &\text{ when } N\geq 3,\\
		1+\gamma \text{ for some } \gamma > 0 &\text{ when } N = 2,\\
		1 &\text{ when } N = 1.
	\end{cases}
\end{equation*}
Then for any $\varepsilon>0$ with $\alpha+\varepsilon<1$ we have
\begin{equation*}
	H(\bm{u}(t)|\bm{u}_\infty) \leq H(\bm{u}_0| \bm{u}_\infty)e^{-C_{\alpha, \varepsilon}\pa{1+t}^{1-\alpha-\varepsilon}} \quad \text{ for all } \quad t\geq 0.
\end{equation*}
Moreover, if $\mu = 0$, which is equivalent to
\begin{equation}\label{uniform_Linf_bound}
	\sup_{i=1,\ldots, 4}\sup_{t\geq 0} \norm{u_i(t)}_{L^\infty(\Omega)} < +\infty,
\end{equation}
we have the full exponential convergence
\begin{equation*}
	H(\bm{u}(t)|\bm{u}_\infty) \leq Ce^{-\lambda_0 t} \quad \text{ for all } \quad t\geq 0,
\end{equation*}
for some $\lambda_0 > 0$.
In both cases the constants that govern the convergence are explicit, and depend only on the domain $\Omega$, the dimension $N$, the diffusion coefficients $d_1, d_2, d_3$, the initial masses $M_{13}, M_{23}, M_{14}$, and in the case where $N=2$, also on $\gamma>0$.
\end{theorem}
\begin{proof}
From \eqref{eq:convergence_bound_condition} we have
\begin{equation*}
	\norm{u_i(t)}_{L^\infty(\Omega)} \leq C_{\mu,\alpha}(1+t)^{\mu} \quad \text{ for all } i=1,\ldots, 4,
\end{equation*}
and
\begin{equation*}
	\norm{u_i(t)}_{L^q(\Omega)} \leq C_{\mu,\alpha}(1+t)^{\alpha} \quad \text{ for } \quad i=1,4.
\end{equation*}
This, together with \eqref{eq:entropy_method} from Theorem \ref{thm:entropy_method}, imply that
\begin{equation*}
	\frac{H(\bm{u}(t)|{\bm{u}_\infty})}{D(\bm{u}(t))} \leq K_1\log\pa{C_{\mu,\alpha}(1+t)^\mu + 1}\pa{1+C_{\mu,\alpha}(1+t)^{\alpha}} \leq K\pa{1+t}^{\alpha+\varepsilon}
\end{equation*}
for some $K$ depending on $K_1$, $C_{\mu,\alpha}$, $\mu$, and $\varepsilon>0$. As $K_1$ depends only on parameters of the problem, and conserved quantities, we conclude that $K$ is a constant that is independent of time, and as such
$$-\frac{d}{dt}H\pa{\bm{u}(t)|\bm{u}_\infty}=D\pa{\bm{u}(t)} \geq \frac{H\pa{\bm{u}(t)|\bm{u}_\infty}}{K\pa{1+t}^{\alpha+\varepsilon}}.$$
Thus,
$$H\pa{\bm{u}(t)|\bm{u}_{\infty}}  \leq H\pa{\bm{u}_0|\bm{u}_{\infty}}e^{-\frac{K}{1-\alpha-\epsilon}\pa{1+t}^{1-\alpha-\epsilon}}, $$
which is the desired estimate.

\medskip
In the case where $\mu = 0$ we have for any $1\leq p \leq \infty$,
\begin{equation*}
	\sup_{i=1,\ldots, 4}\sup_{t\geq 0}\norm{u_i(t)}_{L^p(\Omega)} \leq C,
\end{equation*}
since the domain is bounded and as such 
$$\norm{u}_{L^p\pa{\Omega}}\leq \abs{\Omega}^{\frac{1}{p}}\norm{u}_{L^\infty\pa{\Omega}}.$$
Using Theorem \ref{thm:entropy_method} we conclude that
\begin{equation*}
\frac{H(\bm{u}(t)|\bm{u}_\infty)}{D(\bm{u}(t))} \leq \lambda_0
\end{equation*}
for some $\lambda_0>0$, which leads to 
\begin{equation*}
	\frac{d}{dt}H(\bm{u}(t)|u_\infty) = D(\bm{u}(t)) \leq  -\lambda_0 H(\bm{u}(t)|\bm{u}_\infty)
\end{equation*}
and consequently
\begin{equation*}
	H(\bm{u}(t)|\bm{u}_\infty) \leq e^{-\lambda_0t}H(\bm{u}_0|\bm{u}_\infty).
\end{equation*}
\end{proof}

\section{Global existence of solutions}\label{sec:bounds}
Now that we have concluded our investigation of the entropy method, we turn our attention to the systematic study of \eqref{eq:sys}: the existence of global bounded solutions to it, and time estimates on appropriately relevant norms.
 
Let us consider our system of equations more closely: the first three equations, which describe the evolution of $\br{u_i}_{i=1,\dots,3}$, include diffusion. As such, if the reaction relents in the end, one expects that the solutions to such equations would remain bounded uniformly in time. The last equation, for $u_4$, is more of an ODE in its nature than a full blown PDE, and can provide us with no further regularity on the solution than the initial datum gives. Moreover, the non-negativity of the solutions implies
$$\partial_t u_4(x,t) \leq u_1(x,t) u_2(x,t) \leq \norm{u_1(t)}_{L^\infty\pa{\Omega}}\norm{u_2(t)}_{L^\infty\pa{\Omega}},$$
which tells us that, even when $u_1$ and $u_2$ are uniformly bounded in time, $u_4$ might grow linearly with $t$.

The above considerations are, unfortunately, more of an intuition than an actual fact. In fact, since potential growth of $u_4$ can considerably affect $\br{u_i}_{i=1,\dots,3}$, we cannot achieve the desired uniform bound of $L^\infty(\Omega)$ norm on them so readily. Interestingly, we \emph{are able} to achieve such a uniform bound on $u_3$, and additionally we can estimate the growth in time for the $L^p(\Omega)$ norm of $u_1$ and $u_2$, for any $1<p<\infty$. These results are expressed in the next two propositions, which are the main results of this section.
\begin{proposition}[$N=1,2$]\label{Global1D}
Consider a bounded domain $\Omega\subset \R^N$, $N = 1,2$ with smooth boundary, $\partial \Omega$. Then, for any non-negative and smooth initial data $\bm{u}_0$  that satisfies the compatibility condition, \eqref{eq:sys} has a unique classical global solution which satisfies
\begin{equation}\label{eq:bound_u_3}
\sup_{t\geq 0}\|u_3(t)\|_{L^\infty(\Omega)} \leq C < +\infty,
\end{equation}
and 
\begin{equation}\label{eq:bound_u124}
	\norm{u_1(t)}_{L^\infty(\Omega)} + \norm{u_2(t)}_{L^\infty(\Omega)} + \norm{u_4(t)}_{L^\infty(\Omega)} \leq C(1+t)^{\mu}
\end{equation}
for some  {\normalfont explicit} $C>0$ and $\mu \geq 0$. Moreover, for any $0<\epsilon<\mu$, there exist  {\normalfont explicit} $\gamma>0$ and  $C_{\gamma,\epsilon}>0$ such that
\begin{equation}\label{eq:bound_u_1_u_2}
\norm{u_1(t)}_{L^{1+\gamma}(\Omega)} + \norm{u_2(t)}_{L^{1+\gamma}(\Omega)} + \norm{u_4(t)}_{L^{1+\gamma}(\Omega)} \leq C_{\gamma,\varepsilon}(1+t)^{\varepsilon}
\end{equation}
for all $t\geq 0$. All constants here depend only on the parameters of the problem and the initial datum.
\end{proposition}

\begin{proposition}[$N\geq 3$]\label{Global3D}
Consider a bounded domain $\Omega\subset \R^N$, $N \geq 3$ with smooth boundary, $\partial \Omega$. Assume moreover that
\begin{equation}\label{quasi-uniform}
	\frac{|d_i - d_3|}{d_i+d_3} < \frac{1}{C_{\mathrm{mr},p_0'}}
\end{equation}
with $i=1$ or $i=2$, and $p_0' > 1$ such that
\begin{equation*}
	p_0 = \frac{p_0'}{p_0' - 1} > \frac{N+2}{2},
\end{equation*}
where $C_{\mathrm{mr},p_0'}$ is a fixed constant that relates to the maximal regularity property of the heat equation (see Lemma \ref{lem_mr}). Then, for any non-negative and smooth initial data $\bm{u}_0$  that satisfies the compatibility condition, \eqref{eq:sys} has a unique classical global solution which satisfies
\begin{equation}\label{3D_u3}
	\sup_{t\geq 0}\norm{u_3(t)}_{L^\infty(\Omega)} \leq C <+\infty,
\end{equation}
and, there exist {\normalfont explicit} constants $\mu \geq 0$ and  $0\leq\alpha < 1$ such that
\begin{equation}\label{3D_u124_Linfty}
	\norm{u_1(t)}_{L^{\infty}(\Omega)} + \norm{u_2(t)}_{L^{\infty}(\Omega)} + \norm{u_4(t)}_{L^{\infty}(\Omega)} \leq C(1+t)^\mu,
\end{equation}
\begin{equation}\label{3D_u124}
	\norm{u_1(t)}_{L^{\frac N2}(\Omega)} + \norm{u_2(t)}_{L^{\frac N2}(\Omega)} + \norm{u_4(t)}_{L^{\frac N2}(\Omega)} \leq C(1+t)^\alpha \quad \text{ for all } \quad t\geq 0.
\end{equation}
All constants here depend only on the parameters of the problem and the initial datum.
\end{proposition}

The rest of this section is organised as follow: in subsection \ref{auxiliary} we provide some auxiliary results which we will use in subsequent subsections. Subsections \ref{3D_proof} and \ref{12D_proof} will be devoted to proving Propositions \ref{Global3D} and \ref{Global1D} respectively.

\medskip
Finally, we recall that we {\it always} denote by $C_T$ a generic constant depending {\it at most algebraically} on the time horizon $T>0$, which can be different from line to line (or even in the same line).

\subsection{Auxiliary results}\label{auxiliary}
We emphasise that all results of this subsection are valid for all $N\geq 1$.
\begin{lemma}\label{lem:polynomial}
	Let $d>0$, $T>0$ and $f\in L^p(\Omega_{T})$ for $1<p<\infty$. Assume, in addition,  that $\norm{f}_{L^p(\Omega_{T})} \leq C_{T}$ and let $u$ be the solution the heat equation
	\begin{equation}\label{heat_equation}
	\begin{cases}
		u_t(x,t) - d\Delta u(x,t) = f, &(x,t)\in \Omega_{T},\\
		\nabla u(x,t) \cdot \nu(x) = 0, &(x,t)\in  \partial\Omega\times(0,T),\\
		u(x,0) = u_0(x), &x\in\Omega,
	\end{cases}
	\end{equation}
with initial datum $u_0\in L^\infty(\Omega)$.
	\begin{itemize}
		\item[(i)] If $p< (N+2)/2$, then 
		\begin{equation*}
			\norm{u}_{L^s(\Omega_{T})} \leq  C_{T} \quad \text{ for all } \quad s < \frac{(N+2)p}{N+2-2p}.
		\end{equation*}
		\item[(ii)] If $p = (N+2)/2$, then
		\begin{equation*}
			\norm{u}_{L^s(\Omega_{T})}   \leq C_{T} \quad \text{ for all } \quad s<+\infty.
		\end{equation*}
		\item[(iii)] If $p > (N+2)/2$, then 
		\begin{equation*}
			\norm{u}_{L^\infty(\Omega_{T})} \leq C_{T}.
		\end{equation*}
	\end{itemize}
\end{lemma}
\begin{proof}
	The proofs of part (i) and (ii) can be found in \cite[Lemma 3.3]{CDF14}, while part (iii) is proved in \cite[Lemma 4.6]{Tan18}.
\end{proof}
The following lemma is in the same spirit with Lemma \ref{lem:polynomial} but on any time interval $(\tau, T)$, which will be useful in studying the uniform boundedness of solutions.
\begin{lemma}\label{cor:embedding_for_N}
Let $d>0$, $0\leq \tau < T$ and $0\leq \theta \in L^p(\Omega_{\tau,T})$ for some $1<p<\infty$. Let $\psi$ be the solution to 
\begin{equation}\label{heat_equation_bw}
\begin{cases}
\partial_t \psi(x,t) + d\Delta \psi(x,t)  = -\theta(x,t), &(x,t)\in \Omega_{\tau, T},\\
\nabla \psi(x,t) \cdot \nu(x) = 0, &(x,t)\in  \partial\Omega\times(\tau,T),\\
\psi(x,T) = 0, &x\in\Omega.
\end{cases}
\end{equation}
Then
\begin{equation*}
	\psi(x,t)\geq 0 \quad \text{ a.e. in } \quad \Omega_{\tau,T}.
\end{equation*}
Moreover, we have the following estimates.
	\begin{itemize}
		\item[(i)] If $p< (N+2)/2$, then 
		\begin{equation}\label{item:less}
			\norm{\psi}_{L^s(\Omega_{\tau,T})} \leq C(T-\tau,\Omega,d,p,s)\norm{\theta}_{L^p(\Omega_{\tau,T})} \quad \text{ for all } \quad s < \frac{(N+2)p}{N+2-2p}.
		\end{equation}
		\item[(ii)] If $p = (N+2)/2$, then
		\begin{equation}\label{item:equal_N}
			\norm{\psi}_{L^s(\Omega_{\tau,T})} \leq C(T-\tau,\Omega,d,p,s)\norm{\theta}_{L^p(\Omega_{\tau,T})} \quad \text{ for all } \quad s<+\infty.
		\end{equation}
		\item[(iii)]If $p > (N+2)/2$, then 
		\begin{equation}\label{item:greater_N} 
			\norm{\psi}_{L^\infty(\Omega_{\tau,T})} \leq C(T-\tau,\Omega,d,p)\norm{\theta}_{L^p(\Omega_{\tau,T})}.
		\end{equation}
	\end{itemize}
\begin{proof}
	At the first glance, equation \eqref{heat_equation_bw} looks like a backward heat equation. However, defining $\widetilde{\psi}(x,t) = \psi(x,T+\tau-t)$, we see that $\widetilde{\psi}$ solves the usual heat equation
	\begin{equation*}
	\begin{cases}
		\partial_t \widetilde{\psi}(x,t) - d\Delta\widetilde{\psi}(x,t) = \widetilde{\theta}(x,t), &(x,t)\in \Omega_{\tau,T},\\
		\nabla \widetilde{\psi}(x,t)\cdot \nu(x) = 0, &(x,t)\in \partial\Omega\times (\tau,T),\\
		\widetilde{\psi}(x,\tau) = 0, &x\in\Omega,
	\end{cases}
	\end{equation*}
where $\widetilde{\theta}(x,t)=\theta(x,T+\tau-t)$. Since $\theta \geq 0$, the non-negativity of $\widetilde{\psi}$, and hence of $\psi$, follows from the maximum principle. By the maximal regularity principle for the heat equation, see \cite{LSU68}, we have
	\begin{equation*}
		\norm{\widetilde{\psi}}_{W_{p,\Omega_{\tau,T}}^{(2,1)}} \leq C(T-\tau,\Omega,d,p)\norm{\widetilde{\theta}}_{L^p(\Omega_{\tau,T})},
	\end{equation*}
	where $W_{p,\Omega_{\tau,T}}^{(2,1)}$ is the Banach space equipped with the norm
	\begin{equation*}
		\norm{\xi}_{W_{p,\Omega_{\tau,T}}^{(2,1)}} = \sum_{2r+|\beta| \leq 2}\norm{\partial_t^{r}\partial_{x}^{\beta}\xi}_{L^p(\Omega_{\tau,T})}.
	\end{equation*}
	Using the embedding
	\begin{equation*}
		W_{p,\Omega_{\tau,T}}^{(2,1)} \hookrightarrow L^s(\Omega_{\tau,T})
	\end{equation*}
	from  \cite{LSU68}, where 
	$$1\leq s < \frac{(N+2)p}{N+2-2p} \quad\text{if} \quad  p<\frac{N+2}{2},$$
	$$ 1\leq s < \infty\quad\text{ arbitrary if}\quad p = \frac{N+2}{2},$$
	and
	$$s=\infty \quad \text{if}\quad p>\frac{N+2}{2},$$ 
	we obtain the desired estimates for $\psi$.\\
\end{proof}
\end{lemma}
The following lemmas are necessary for the improved duality method.
\begin{lemma}\label{lem_mr}
	Let $0\leq \tau < T$ and $\theta \in L^p(\Omega_{\tau,T})$ for some $1<p<\infty$. Let $\psi$ be the solution to the equation
\begin{equation}\label{normalized_diffusion}
		\begin{cases}
			\partial_t \psi(x,t) + \Delta \psi(x,t)= -\theta(x,t), &(x,t)\in\Omega_{\tau,T},\\
			\nabla\psi(x,t) \cdot \nu(x) = 0, &(x,t)\in \partial\Omega\times(\tau,T),\\
			\psi(x,T) = 0, &x\in\Omega.
		\end{cases}
	\end{equation}

	Then there exists an optimal constant $C_{\mathrm{mr},p}$ depending only on $p$, the domain $\Omega$, and the dimension $N$ such that the following \emph{maximal regularity} holds
	\begin{equation}\label{Cmr}
		\norm{\Delta \psi}_{L^p(\Omega_{\tau,T})} \leq C_{\mathrm{mr},p}\norm{\theta}_{L^p(\Omega_{\tau,T})}.
	\end{equation}
\end{lemma}
\begin{proof}
	The proof for the regular heat equation can be found in \cite[Theorem 1]{Lam87}, and in the same manner we have shown the validity of the previous lemma (by defining $\widetilde{\psi}$), the result follows.
\end{proof}
\begin{lemma}\label{lem_mr_rescaled}
	Let $d>0$, $0\leq \tau < T$ and $\theta \in L^p(\Omega_{\tau,T})$ for some $1<p<\infty$. Let $\psi$ be the solution to \eqref{heat_equation}. Then
	\begin{equation}\label{b5}
		\norm{\Delta\psi}_{L^p(\Omega_{\tau,T})} \leq \frac{C_{\mathrm{mr},p}}{d}\norm{\theta}_{L^p(\Omega_{\tau,T})}
	\end{equation}
	where $C_{\mathrm{mr},p}$ is the constant in \eqref{Cmr}.
\end{lemma}
\begin{proof}
	Define $\wh{\psi}(x,t) = \psi(x,t/d)$ we have
\begin{equation*}
		\begin{cases}
			\partial_t\wh{\psi}(x,t) - \Delta \psi(x,t) = \frac 1d\wh{\theta}(x,t), &(x,t)\in \Omega_{d\tau, dT},\\
			\nabla \wh{\psi}(x,t) \cdot \nu (x)= 0, &(x,t)\in \partial \Omega\times (d\tau, dT),\\
			\wh{\psi}(x,d\tau) = 0, &x\in\Omega,
		\end{cases}
	\end{equation*}
	where $\wh{\theta}(x,t) = \theta(x,t/d)$. From Lemma \ref{lem_mr}, we have 
	\begin{equation*}
		\|\Delta\wh{\psi}\|_{L^p(\Omega_{d\tau, dT})} \leq \frac{C_{\rm{mr},p}}{d}\|\wh{\theta}\|_{L^p(\Omega_{d\tau, dT})}
	\end{equation*}
	or equivalently
	\begin{equation*}
	\int_{d\tau}^{dT}\int_{\Omega}|\Delta\wh{\psi}(x,t)|^pdxdt \leq \pa{\frac{C_{\rm{mr},p}}{d}}^p\int_{d\tau}^{dT}\int_{\Omega}|\wh{\theta}(x,t)|^pdxdt
	\end{equation*}
	Making the change of variable $s = t/d$ yields
	\begin{equation*}
		d\|\Delta \psi\|_{L^p(\Omega_{\tau,T})}^p \leq \frac{C_{\rm{mr},p}^p}{d^{p-1}}\|\theta\|_{L^p(\Omega_{\tau,T})}^p,
	\end{equation*}
	which is \eqref{b5}.
\end{proof}
\begin{lemma}\label{elementary}
	Let $\{a_n\}_{n\in \mathbb N}$ be a non-negative sequence. Define $\mathfrak N = \{n\in \mathbb N: \; a_{n-1} \leq a_n \}$. If there exists a constant $C$ independent of $n$ such that
	\begin{equation*}
		a_n \leq C \quad \text{ for all } \quad n \in \mathfrak N,
	\end{equation*}
	then
	\begin{equation*}
		a_n \leq \max\{a_0; C\} \quad \text{ for all } \quad n \in \mathbb N.
	\end{equation*}
\end{lemma}
\begin{proof}
	The proof of this lemma is straightforward. We merely mentioned it explicitly as we will use it several times.
\end{proof}
\begin{lemma}[A Gronwall-type inequality]\label{Gronwall}
Assume that $y(t)$ is a non-negative function that satisfies
\begin{equation*}
	y'(t) \leq \beta(t) + \kappa y(t)^{1-r}, \quad \forall t> 0,
\end{equation*}
for $\kappa>0$, $r\in (0,1)$, and a non-negative function $\beta(t)$. Then
\begin{equation*}
y(t) \leq C\rpa{y(0) + \int_0^t\beta(s)ds + t^{\frac 1r}}, \quad \forall t>0,
\end{equation*}
where the constant $C$ depends only on $\kappa$ and $r$.
\end{lemma}
\begin{proof}
 Using Young's inequality, we find that for \emph{any} given $\delta>0$,
	\begin{equation*}
		y^{1-r} \leq \delta y + r\pa{\frac{1-r}{\delta}}^{\frac{1-r}{r}}.
	\end{equation*}
It then follows that
	\begin{equation*}
		y' \leq \kappa\delta y + \beta(t) + \kappa r\pa{\frac{1-r}{\delta}}^{\frac{1-r}{r}} =: \kappa\delta y + \gamma_{\delta}(t).
	\end{equation*}
which implies that,
	\begin{equation*}
		y(t) \leq y(0)e^{\kappa\delta t} + \int_0^t\gamma_{\delta}(s)e^{\kappa\delta(t-s)}ds.
	\end{equation*}
	Choosing $\delta = (\kappa t)^{-1}$ we obtain
	\begin{align*}
		y(t) &\leq ey(0) + e\int_0^t\pa{\beta(s) + \kappa^{\frac 1r}r\pa{1-r}^{\frac{1-r}{r}}t^{^{\frac{1-r}{r}}}}ds\\
		&\leq ey(0) + e\int_0^t\beta(s)ds + e\kappa^{\frac 1r}r(1-r)^{\frac{1-r}{r}}t^{\frac 1r}\\
		&\leq \max\br{e,e\kappa^{\frac 1r}r(1-r)^{\frac{1-r}{r}}}\pa{y(0) + \int_0^t\beta(s)ds + t^{\frac 1r}}
	\end{align*}
	which finishes the proof.
\end{proof}
\subsection{Proof of Proposition \ref{Global3D}}\label{3D_proof}
The proof of Proposition \ref{Global3D} is quite involved. We shall show the existence and uniqueness of a classical solution, together with \eqref{3D_u124_Linfty}, in Proposition \ref{global3D}. \eqref{3D_u3} will be shown in Lemma \ref{lem:L_infty_bound_on_u_3}, and Lemma \ref{3D_u4LN2} will prove \eqref{3D_u124}.\\
The first step in achieving a proof is the following improved duality lemma.
\begin{lemma}\label{lem_b1}
	Assume that condition \eqref{quasi-uniform} is satisfied. Then, any classical solution to \eqref{eq:sys} on $\Omega_{T}$ satisfies
	\begin{equation}\label{u123_Lp}
				\|u_i\|_{L^{p_0}(\Omega_T)} \leq C_T \quad \text{ for } \quad i=1,2,3,
		\end{equation}
	where $p_0 = \frac{p_0'}{p_0'-1}$ with $p_0'$ is as in \eqref{quasi-uniform}, and the constant $C_T$ depends on the initial data $\br{u_{i,0}}_{i=1,\dots,3}$ and grows at most algebraically in $T$.
\end{lemma}
\begin{proof}
	Without loss of generality, we assume that \eqref{quasi-uniform} holds for $i=1$. Fix $0\leq \theta\in L^{p_0'}(\Omega_T)$ and let $\psi$ be the solution to \eqref{heat_equation_bw} with $\tau = 0$, and $d = \frac{d_1+d_3}{2}$. Thanks to Lemma \ref{lem_mr_rescaled} we know that
		\begin{equation}\label{2star}
			\|\Delta \psi\|_{L^{p_0'}(\Omega_T)} \leq \frac{2C_{\mathrm{mr},p_0'}}{d_1+d_3}\norm{\theta}_{L^{p_0'}(\Omega_T)}.
		\end{equation}
		Using integration by parts together with the fact that
\begin{equation}\label{eq:u_1+u_3}
\partial_t \pa{u_1(x,t)+u_3(x,t) }= d_1\Delta u_1(x,t)+d_3\Delta u_3(x,t)
\end{equation}
		 we find that
		\begin{equation}\label{b1}
		\begin{aligned}
		&\int_0^T\int_{\Omega}(u_1(x,t)+u_3(x,t))\theta(x,t) dxdt\\
		&= \int_0^T\int_{\Omega}(u_1(x,t)+u_3(x,t))(-\partial_t \psi(x,t) - d\Delta \psi(x,t))dxdt\\
		&= \int_{\Omega}(u_{1,0}(x)+u_{3,0}(x))\psi(x,0)dx + (d_1 - d)\int_0^T\int_{\Omega}u_1(x,t)\Delta \psi(x,t) dxdt\\
		&+ (d_3 - d)\int_0^T\int_{\Omega} u_3(x,t)\Delta \psi(x,t) dxdt \leq \norm{u_{1,0}+u_{3,0}}_{L^{p_0}(\Omega)}\norm{\psi(0)}_{L^{p_0'}(\Omega)} \\&+ \frac{|d_1 - d_3|}{2}\norm{u_1 + u_3}_{L^{p_0}(\Omega_T)}\norm{\Delta \psi}_{L^{p_0'}(\Omega_T)}\\
		&\leq \norm{u_{1,0}+u_{3,0}}_{L^{p_0}(\Omega)}\norm{\psi(0)}_{L^{p_0'}(\Omega)} + \frac{|d_1 - d_3|}{d_1+d_3}C_{\rm{mr}, p_0'}\norm{u_1 + u_3}_{L^{p_0}(\Omega_T)}\norm{\theta}_{L^{p_0'}(\Omega_T)}
		\end{aligned}
		\end{equation}
		where the last inequality followed from \eqref{2star}. Since $\psi$ also satisfies
		\begin{equation*}
			\norm{\partial_t \psi}_{L^{p_0'}(\Omega_T)} \leq \frac{d_1+d_3}{2}\norm{\Delta \psi}_{L^{p_0'}(\Omega_T)} + \norm{\theta}_{L^{p_0'}(\Omega_T)} \leq \pa{C_{\rm{mr},p_0'} + 1}\norm{\theta}_{L^{p_0'}(\Omega_T)},
		\end{equation*}
		where we used \eqref{2star} again, we find that
		\begin{equation*}
			\norm{\psi(0)}_{L^{p_0'}(\Omega)}^{p_0'} = \int_{\Omega}\abs{\int_0^T\partial_t\psi(x,s) ds}^{p_0'}dx \leq T^{\frac{p_0^\prime}{p_0}}\norm{\partial_t\psi}_{L^{p_0'}(\Omega_T)}^{p_0'},
		\end{equation*}
		and as such
		\begin{equation*}
			\norm{\psi(0)}_{L^{p_0'}(\Omega)} \leq T^{1/p_0}\pa{C_{\rm{mr}, p_0'} + 1}\norm{\theta}_{L^{p_0'}(\Omega_T)}.
		\end{equation*}
		Inserting this into \eqref{b1}, and using duality we obtain
		\begin{align*}
			\norm{u_1+u_3}_{L^{p_0}(\Omega_T)} \leq \norm{u_{1,0}+u_{3,0}}_{L^{p_0}(\Omega)}\pa{C_{\rm{mr}, p_0'} + 1}T^{1/p_0}\\
			+ \frac{|d_1 - d_3|}{d_1+d_3}C_{\rm{mr}, p_0'}\norm{u_1+u_3}_{L^{p_0}(\Omega_T)}.
		\end{align*}
Thus, since \eqref{quasi-uniform} is valid, we have that
		\begin{equation*}
			\norm{u_1+u_3}_{L^{p_0}(\Omega_T)} \leq  \frac{C_{\rm{mr}, p_0'} + 1}{1-\frac{|d_1 - d_3|}{d_1+d_3}C_{\rm{mr}, p_0'}}\norm{u_{1,0}+u_{3,0}}_{L^{p_0}(\Omega)}T^{1/p_0}:=C_T.
		\end{equation*}
		Since $u_1$ and $u_3$ are non-negative, this estimate also holds for $u_1$ and $u_3$ individually. \\
		To deal with $u_2$ we remember that \eqref{eq:sys} implies that
		\begin{equation*}
			\partial_t(u_2 + u_3)(x,t) - \Delta(d_2u_2 + d_3u_3)(x,t) = 0
		\end{equation*}
		and therefore, using \cite[Lemma 33.3, page 291]{QS07}, we find that
		\begin{equation*}
			\norm{u_2}_{L^{p_0}(\Omega_T)} \leq C\pa{\norm{u_{2,0}+u_{3,0}}_{L^{p_0}(\Omega)} + \norm{u_3}_{L^{p_0}(\Omega_T)}} \leq C_T,
		\end{equation*}
		which completes the proof.
\end{proof}
\begin{proposition}\label{global3D}
Assume the quasi-uniform condition \eqref{quasi-uniform} for the diffusion coefficients holds. Given any smooth, bounded, and non-negative initial data  that satisfies the compatibility condition, \eqref{eq:sys} admits a unique classical global, bounded solution which obeys the estimate \eqref{3D_u124_Linfty}.
\end{proposition}
\begin{proof}
	To begin with we notice that since the non-linearities on the right hand side of \eqref{eq:sys} are locally Lipschitz, the local existence of a classical bounded solution on a maximal interval $(0,T_{\max})$ is a classical result (see, \cite{Ama85}). Furthermore, as the non-linearities are quasi-positive, that is $f_i(\bm{u})\geq 0$ as long as $\bm{u} \in \mathbb R_+^4$ and $u_i = 0$, we know that if the initial data is non-negative, then when a solution exists it must also be non-negative (see e.g. \cite{Pie10}). Moreover,
	\begin{equation}\label{global_criterion}
		\text{ if } \; \sup_{i=1,\ldots, 4}\limsup_{t\to T_{\max}^-}\|u_i(t)\|_{L^\infty(\Omega)}<+\infty.
	\end{equation}
	then $T_{\max}=\infty$, and the solution is in fact global. 
	From Lemma \ref{lem_b1} we have that for all $0<T<T_{\max}$,
	\begin{equation*}
		\norm{u_i}_{L^{p_0}(\Omega_T)} \leq C_T \quad \text{ for } \quad i=1,2,3
	\end{equation*}
 where $C_T$ \textit{grows at most algebraically} in $T$, and $p_0 = \frac{p_0'}{p_0'-1} > \frac{N+2}{2}$. We use that together with our system of PDEs to bootstrap our norm estimate until we reach $L^\infty$.
 
 The non-negativity of the solution implies that
	\begin{equation*}
		\partial_t u_3(x,t) -d\Delta u_3(x,t) \leq u_1(x,t)u_2(x,t).
	\end{equation*}
Moreover, since $u_1u_2 \in L^{\frac{p_0}{2}}(\Omega_T)$ we can use the comparison principle for the heat equation (see, for instance, \cite[Theorem 2.2.1]{Pao12}) and the maximal regularity from Lemma \ref{lem:polynomial} to conclude that
	\begin{equation}\label{b2}
		\|u_3\|_{L^{p_1-}(\Omega_T)} \leq C_{T,p_1-}
	\end{equation}
	with $p_1 = \frac{(N+2)\frac{p_0}{2}}{N+2-p_0}$ if $N+2 > p_0$ and $p_1 = +\infty$ otherwise, where above means that for any $s<p_1\leq +\infty$
	$$\norm{u_3}_{L^{s}(\Omega_T)} \leq C_{T,s}.$$
	Using the fact that
	\begin{equation*}
		\partial_t(u_1 + u_3)(x,t)- \Delta(d_1u_1 + d_3u_3) (x,t)= 0, \quad  \quad \partial_t(u_2 + u_3)(x,t) - \Delta(d_2u_2 + d_3u_3)(x,t) = 0
	\end{equation*}
	and applying \cite[Lemma 33.3]{QS07} again, we conclude, due to \eqref{b2}, that
	\begin{equation*}
		\norm{u_1}_{L^{p_1-}(\Omega_T)} \leq C_{T,p_1-} \quad \text{ and } \quad \norm{u_2}_{L^{p_1-}(\Omega_T)} \leq C_{T,p_1-}.
	\end{equation*}
	If $p_1=+\infty$, we stop here. Otherwise, we repeat this process and construct a sequence $\{p_n\}_{n\geq 0}$ such that 
$$p_{n+1}=\begin{cases} \frac{(N+2)\frac{p_n}{2}}{N+2-p_n} & N+2 > p_n\\
+\infty & \text{otherwise},
\end{cases}$$	
and
	\begin{equation*}
		\norm{u_i}_{L^{p_n-}(\Omega_T)} \leq C_T \quad \text{ for } \quad i=1,2,3 \quad \text{ and } \quad n\ge 1.
	\end{equation*}
	We claim that since $p_0>\frac{N+2}{2}$, there must exist $n_0\geq 1$ such that $p_{n_0} \geq N+2$. Indeed, assume that $p_n < N+2$ for all $n\geq 1$. Then by definition, we have that
	\begin{equation*}
		\frac{p_{n+1}}{p_n} = \frac{N+2}{2(N+2-p_n)}.
	\end{equation*}	
which will be shown to be greater or equal to $1$. Since $p_0 > \frac{N+2}{2}$ we have that
$$\frac{p_1}{p_0}=\frac{N+2}{2(N+2-p_0)} > \frac{N+2}{2(N+2-\frac{N+2}{2})}=1.$$
Continuing by induction we assume that $\frac{p_{n}}{p_{n-1}}>1$ and conclude that
$$\frac{p_{n+1}}{p_n}=\frac{N+2}{2(N+2-p_n)} > \frac{N+2}{2(N+2-p_{n-1})}=\frac{p_n}{p_{n-1}}>1.$$
Thus, $\br{p_n}_{n\in\N}$ is increasing and satisfies
$$p_0 < p_n < \frac{N+2}{2},\quad \forall n\in\N.$$
This implies that $\br{p_n}_{n\in\N}$ converges to a finite limit $p>p_0$, that must satisfy
$$1=\frac{N+2}{2\pa{N+2-p}},$$
which is, of course, impossible. We have thus found $n_0 \geq 1$ such that $p_{n_0} \geq N+2$. Applying Lemma \ref{lem:polynomial} one more time to
\begin{equation}\label{eq:u3_eq_for_bootstrap}
\partial_t u_3 (x,t)- d_3\Delta u_3 (x,t)\leq u_1(x,t)u_2(x,t),
\end{equation}
now for $u_1u_2 \in L^{\frac{p_{n_0}}{2}}(\Omega_T)$ and $\frac{p_{n_0}}{2} \geq \frac{N+2}{2}$, we find that
	\begin{equation*}
		\norm{u_3}_{L^q(\Omega_T)} \leq C_{q,T} \quad \text{ for all } \quad 1\leq q < +\infty.
	\end{equation*}
	Consequently, like before, 
	\begin{equation*}
		\norm{u_1}_{L^q(\Omega_T)} + \norm{u_2}_{L^q(\Omega_T)}  \leq C_{q,T} \quad \text{ for all } \quad 1\leq q < +\infty.
	\end{equation*}
Looking at \eqref{eq:u3_eq_for_bootstrap}, with the above knowledge, we use the comparison principle and Lemma \ref{lem:polynomial} again, this time with any $q_0>N+2$, to conclude that
	\begin{equation}\label{eq:u3_infty}
		\norm{u_3}_{L^\infty(\Omega_T)} \leq C_{T,q_0}.
	\end{equation}
	To obtain boundedness of $u_1, u_2$ and $u_4$ we first observe that
	\begin{equation}\label{b3}
		u_4(x,t) \leq u_{4,0}(x) + \int_0^tu_1(x,s)u_2(x,s)ds,
	\end{equation}
	which implies for any $1\leq q <\infty$, 
	\begin{equation*}
	\begin{gathered}
		\norm{u_4}_{L^q(\Omega_T)} \leq T\sup_{t\in (0,T)}\norm{u_4(t)}_{L^q(\Omega)} \leq T\pa{\norm{u_{4,0}}_{L^q(\Omega)} + T^{\frac{1}{q^\prime}}\norm{u_1u_2}_{L^q(\Omega_T)}}\\
		\leq T\pa{\norm{u_{4,0}}_{L^q(\Omega)} + T^{\frac{1}{q^\prime}}\pa{\norm{u_1}_{L^{2q}(\Omega_T)}\norm{u_2}_{L^{2q}(\Omega_T)}}}\leq C_{T,q}.
		 \end{gathered}
	\end{equation*}
	Together with \eqref{eq:u3_infty} we find that $u_3u_4 \in L^q(\Omega_T)$ for any $1\leq q<\infty$, with a time bound that grows at most algebraically in $T$ (and depends on $q$). Since
	$$\partial_t u_1 (x,t)- d_1\Delta u_2 (x,t)\leq u_3(x,t)u_4(x,t), $$
	$$\partial_t u_2 (x,t)- d_2\Delta u_2(x,t) \leq u_3(x,t)u_4(x,t) $$
	we use the comparison principle and Lemma \ref{lem:polynomial} to see that 
	\begin{equation*}
		\norm{u_1}_{L^\infty(\Omega_T)}  \leq C_{T}
	\end{equation*}
	and  
	\begin{equation*}
		\norm{u_2}_{L^\infty(\Omega_T)} \leq C_{T}.
	\end{equation*}
	Lastly, the above, together with \eqref{b3}, imply that
	\begin{equation*}
		\norm{u_4}_{L^\infty(\Omega_T)} \leq C_T,
	\end{equation*}
	and we can conclude that \eqref{global_criterion} is satisfied, completing the proof.
\end{proof}
\begin{remark}\label{rem:existence_sol_no_dim}
A careful look at the proofs of Lemma \ref{lem_b1} and Proposition \ref{global3D} show that we have not used any condition on the dimension $N$. Indeed, as long as condition \eqref{quasi-uniform} is satisfied, with $C_{\mathrm{mr},p_0^\prime}$ the appropriate constant from Lemma \ref{lem_mr} (even for $N=1,2$), we can conclude the existence of a global classical solution, that is non-negative when the initial data is non-negative. \\
An immediate consequence of this, which we will use in what follows, is that Remark \ref{rem:mass_bounds} is valid, and we also obtain estimates \eqref{eq:mass_bounds}, i.e.
$$\norm{u_i(t)}_{L^1\pa{\Omega}} \leq M_{ij},\quad\quad i=1,2,\;j=3,4.$$
\end{remark}
To fully show Proposition \ref{Global3D}, we now turn our attention to prove the norm estimates \eqref{3D_u3} and \eqref{3D_u124}. The key idea we will employ to show these estimates is to study the system \eqref{eq:sys} on cylinders $\Omega_{\tau,\tau+1}$ with $\tau \in \R_{+}$, and to find corresponding estimates {\it uniformly in $\tau$}.\\
In what follows we will always assume that $\br{u_i}_{i=1,\dots,4}$ are classical solutions to our system of equations, and will write our lemmas more succinctly. 
\begin{lemma}\label{lem:some_L_p_estimaes_for_123}
Assume that \eqref{quasi-uniform} holds for $i_0\in \{1,2\}$. Then
\begin{equation}\label{eq:some_L_p_estimates_for_123}
\sup_{\tau \geq 0}\pa{\|u_{i_0}\|_{L^{p_0}(\Omega_{{\tau,\tau+1}})}+\|u_3\|_{L^{p_0}(\Omega_{{\tau,\tau+1}})}} \leq C_0,
\end{equation}
where the constant $C_0$ depends only on $p_0^\prime$, the diffusion coefficients $d_1$, $d_2$, $d_3$, the domain $\Omega$, the initial masses $M_{ij}$, and $\br{\norm{u_i}_{L^{\infty}(\Omega_{0,1})}}_{i=1,2,3}$.
\end{lemma}
\begin{proof}
Without loss of generality, we assume that \eqref{quasi-uniform} holds for $i_0=1$.  For any $0\leq \theta \in L^{p_0'}\pa{\Omega_{{\tau,\tau+2}}}$ we set $\psi_{\theta}$ to be the solution of \eqref{heat_equation_bw} with final time $T:= \tau + 2$, diffusion coefficient $d=\frac{d_1+d_3}{2}$ and inhomogeneous source $-\theta$. From Lemma \ref{lem_mr_rescaled} we know that 
\begin{equation}\label{eq:est_for_lapl}
\|\Delta \psi_{\theta}\|_{L^{p_0'}(\Omega_{\tau,\tau+2})} \leq \frac{2C_{\rm{mr},p_0'}}{d_1+d_3}\|\theta\|_{L^{p_0'}(\Omega_{\tau,\tau+2})}.
\end{equation}
Lemma \ref{cor:embedding_for_N}, and the fact that
$$p_0^\prime = \frac{p_0}{p_0-1} < \frac{N+2}{N} \underset{N\geq 3}{<}\frac{N+2}{2}$$
assures us that for any $q<\frac{(N+2)p_0'}{N+2-2p_0'}$
\begin{equation}\label{eq:some_L_p_estimates_proof_I}
\|\psi_{\theta}\|_{L^{q}(\Omega_{\tau,\tau+2})} \leq C(p_0',d,\Omega)\|\theta\|_{L^{p_0'}(\Omega_{\tau,\tau+2})},
\end{equation}
where the constant $C(p_0', d, \Omega)$ in the above is independent of $\tau$. Let $\phi(s)$ be a $C^\infty\pa{\R}$ function such that $\phi(0)=0$, $0\leq\phi\leq 1$ and $\phi\vert_{[1,\infty)}\equiv 1$. Defining the function $\phi_{\tau}(s)=\phi\pa{s+\tau}$ we notice that
\begin{equation}\nonumber
\begin{gathered}
\partial_t(\varphi_\tau(t) u_1(x,t)) - d_1\Delta_x(\varphi_\tau(t) u_1(x,t)) = u_1(x,t)\varphi^\prime_\tau(t) + \phi_\tau(t)\pa{\partial_t u_1(x,t)+d_1\Delta u_1(x,t)}\\
=u_1(x,t)\varphi^\prime_\tau(t) + \phi_\tau(t) f_1\pa{\bm{u}(x,t)}
\end{gathered}
\end{equation}
Using the fact that $\psi_{\theta}(x,\tau+2)=0$ and $\phi_\tau(\tau)u_1(x,\tau)=0$ we see that by integration by parts
\begin{align*}
&\int_\tau^{\tau+2}\int_{\Omega}(\varphi_\tau(t) u_1(x,t))\theta(x,t) dxdt\\
&=-\int_\tau^{\tau+2}\int_{\Omega}(\varphi_\tau(t) u_1(x,t))(\partial_t \psi_{\theta}(x,t) + \pa{d_1+d-d_1}\Delta \psi_{\theta}(x,t))dxdt\\
&=\int_\tau^{\tau+2}\int_{\Omega}\psi_{\theta}(x,t)\pa{\partial_t(\varphi_\tau(t) u_1(x,t)) - d_1\Delta(\varphi_\tau(t) u_1(x,t))}dxdt \\
&\quad + (d_1 - d)\int_\tau^{\tau+2}\int_{\Omega}(\varphi_\tau(t) u_1(x,t))\Delta \psi_{\theta}(x,t) dxdt\\
&= \int_\tau^{\tau+2}\int_{\Omega}\psi_{\theta}(x,t)\pa{\varphi_\tau^\prime (t) u_1(x,t) + \varphi_\tau(t) f_1(\bm{u}(x,t))}dxdt\\
&\quad + (d_1 - d)\int_\tau^{\tau+2}\int_{\Omega}(\varphi_\tau(t) u_1(x,t))\Delta \psi_{\theta}(x,t) dxdt.
\end{align*}
Similarly,
\begin{align*}
&\int_\tau^{\tau+2}\int_{\Omega}(\varphi_\tau(t) u_3(x,t))\theta(x,t) dxdt\\
&= \int_\tau^{\tau+2}\int_{\Omega}\psi_{\theta}(x,t)\pa{\varphi_\tau^\prime (t) u_3(x,t) + \varphi_\tau(t) \underbrace{f_3(\bm{u}(x,t))}_{-f_1\pa{\bm{u}(x,t)}}}dxdt\\
& \quad + (d_3 - d)\int_\tau^{\tau+2}\int_{\Omega}\pa{\varphi_\tau(t) u_3(x,t)}\Delta \psi_{\theta}(x,t) dxdt.
\end{align*}	
Summing these two equalities, and using the non-negativity of the solutions together with the choice $d=\frac{d_1+d_3}{2}$, yields 
\begin{equation}\label{eq:some_L_p_estimates_proof_II}
\begin{aligned}
&\int_{\tau}^{\tau+2}\int_{\Omega}(\varphi_\tau(t)(u_1(x,t)+u_3(x,t)))\theta(x,t) dxdt\\
&\leq \int_{\tau}^{\tau+2}\int_{\Omega}\psi_{\theta}(x,t) \varphi_\tau^\prime (t)(u_1(x,t) + u_3(x,t))dxdt\\
&+ \frac{\abs{d_1-d_3}}{2}\int_{\tau}^{\tau+2}\int_{\Omega}\varphi_\tau(t)(u_1(x,t) + u_3(x,t))\abs{\Delta \psi_{\theta}(x,t)}dxdt=\mathcal{I}+\mathcal{II}.
\end{aligned}
\end{equation}
Using H\"older's inequality together with \eqref{eq:est_for_lapl} we find that
\begin{equation}\label{eq:some_L_p_estimates_proof_III}
\begin{aligned}
\mathcal{II} &\leq \frac{|d_1-d_3|}{2}\|\Delta \psi_{\theta}\|_{L^{p_0'}(\Omega_{\tau,\tau+2})}\|\varphi_\tau(u_1 + u_3)\|_{L^{p_0}(\Omega_{\tau,\tau+2})}\\
&\leq \frac{|d_1-d_3|}{d_1+d_3}C_{\rm{mr},p_0'}\| \theta\|_{L^{p_0'}(\Omega_{\tau,\tau+2})}\|\varphi_\tau(u_1 + u_3)\|_{L^{p_0}(\Omega_{\tau,\tau+2})}
\end{aligned}
\end{equation}
To estimate $\mathcal{I}$, we pick some $1<q^\prime<p_0$ such that its H\"older conjugate, $q = \frac{q'}{q'-1}$, satisfies
\begin{equation}\nonumber
q < \frac{(N+2)p_0'}{N+2-2p_0'}.
\end{equation}
This is possible since $q^\prime < p_0$ if and only if $p_0'<q$, and as $p_0^\prime<\frac{(N+2)p_0'}{N+2-2p_0'}$, we can easily choose $q$ between them. Using Lemma \ref{cor:embedding_for_N} again, we find that
\begin{equation}\nonumber
\|\psi_{\theta}\|_{L^{q}(\Omega_{\tau,\tau+2})} \leq C\pa{p_0',d,\Omega}\|\theta\|_{L^{p_0'}(\Omega_{\tau,\tau+2})},
\end{equation}
and as such, using the fact that $\phi^{\prime}_{\tau}\vert_{\rpa{\tau+1,\tau+2}}=0$ and $|\varphi_\tau'| \leq C$,
\begin{equation}\label{eq:some_L_p_estimates_proof_IV}
\begin{gathered}
\abs{\mathcal{I}}\leq C\pa{p_0',d,\Omega} \|\psi_\theta\|_{L^{q}(\Omega_{\tau,\tau+2})}\|u_1+u_3\|_{L^{q'}(\Omega_{\tau,\tau+1})}\\
 \leq C\pa{p_0',d,\Omega}\|\theta\|_{L^{p_0'}(\Omega_{\tau,\tau+2})}\|u_1+u_3\|_{L^{q'}(\Omega_{\tau,\tau+1})}
\end{gathered}
\end{equation}
Since $1< q^\prime < p_0$, we know that for $0<\alpha<1$ such that 
$$\frac{1}{q^\prime} = \frac{\alpha}{1} + \frac{1-\alpha}{p_0}$$
one has 
\begin{equation}\nonumber
\|f\|_{L^{q^\prime}(\Omega)} \leq \|f\|_{L^1(\Omega)}^{\alpha}\|f\|_{L^{p_0}(\Omega)}^{1-\alpha},
\end{equation}
which, together with \eqref{eq:some_L_p_estimates_proof_IV}, yield the estimate 
\begin{equation}\label{eq:some_L_p_estimates_proof_V}
\begin{gathered}
\abs{\mathcal{I}}\leq  C\pa{p_0',d,\Omega}\|\theta\|_{L^{p_0'}(\Omega_{\tau,\tau+2})}M_{13}^\alpha\|u_1+u_3\|^{1-\alpha}_{L^{p_0}(\Omega_{\tau,\tau+2})},
\end{gathered}
\end{equation}
where we have used the non-negativity of the solutions, and the mass conservation property again. By inserting \eqref{eq:some_L_p_estimates_proof_III} and \eqref{eq:some_L_p_estimates_proof_V} into \eqref{eq:some_L_p_estimates_proof_II} we obtain
\begin{equation}\nonumber
\begin{aligned}
&\int_{\tau}^{\tau+2}\int_{\Omega}(\varphi_\tau(t)(u_1(x,t)+u_3(x,t)))\theta(x,t) dxdt \\
&\leq \|\theta\|_{L^{p_0'}(\Omega_{\tau,\tau+2})}\left(C\|u_1+u_3\|_{L^{p_0}(\Omega_{\tau,\tau+1})}^{1-\alpha} +\frac{|d_1-d_3|}{d_1+d_3}C_{\rm{mr},p_0'}\|\varphi_\tau(u_1+u_3)\|_{L^{p_0}(\Omega_{\tau,\tau+2})}\right)
\end{aligned}
\end{equation}
with $C>0$ depending only on $p_0'$, $d$, $\Omega$ and $M_1(0)+M_3(0)$. As $\theta$ was arbitrary, we conclude that
\begin{equation}\nonumber
\|\varphi_\tau (u_1+u_3)\|_{L^{p_0}(\Omega_{\tau,\tau+2})} \leq C\|u_1+u_3\|_{L^{p_0}(\Omega_{\tau,\tau+1})}^{1-\alpha} + \frac{|d_1-d_3|}{d_1+d_3}C_{\rm{mr},p_0'}\|\varphi_\tau(u_1+u_3)\|_{L^{p_0}(\Omega_{\tau,\tau+2})},
\end{equation}
which translates to 
\begin{equation}\label{eq:some_L_p_estimates_proof_VI}
\|\varphi_\tau (u_1+u_3)\|_{L^{p_0}(\Omega_{\tau,\tau+2})} \leq \frac{C}{1-\frac{|d_1-d_3|}{d_1+d_3}C_{\rm{mr},p_0'}}\|u_1+u_3\|_{L^{p_0}(\Omega_{\tau,\tau+1})}^{1-\alpha}
\end{equation}
thanks to \eqref{quasi-uniform}.

\medskip
To show the uniform in $\tau$ boundedness of $\norm{u_1+u_3}_{L^{p_0}(\Omega_{\tau,\tau+1})}$ it will be enough to show the boundedness of the sequence
$$a_n=\norm{u_1+u_3}_{L^{p_0}(\Omega_{n,n+1})}$$
where $n\in \N \cup\br{0}$, as
$$\norm{u_1+u_3}_{L^{p_0}(\Omega_{\tau,\tau+1})}\leq a_{\rpa{\tau}}+a_{\rpa{\tau}+1}.$$
We start by noticing that \eqref{eq:some_L_p_estimates_proof_VI}, the fact that $\phi_\tau\vert_{\rpa{\tau+1,\tau+2}}\equiv 1$, and the fact that $\phi_\tau,u_1$ and $u_3$ are non-negative, imply that
\begin{equation}\nonumber 
a_{n+1} \leq \|\varphi_n (u_1+u_3)\|_{L^{p_0}(\Omega_{n,n+2})} \leq \frac{C}{1-\frac{|d_1 - d_3|}{d_1+d_3}C_{\rm{mr}, p_0'}}\|u_1+u_3\|_{L^{p_0}(\Omega_{n,n+1})}^{1-\alpha}=\mathcal{C} a_n^{1-\alpha}.
\end{equation}
Thus if $a_n \leq a_{n+1}$ we must have that
$$a_n \leq \mathcal{C}a_n^{1-\alpha},$$
resulting in $a_n \leq \mathcal{C}^{\frac{1}{\alpha}}.$ From this we infer that
$$a_{n+1} \leq \mathcal{C}\pa{\mathcal{C}}^{\frac{1-\alpha}{\alpha}}.$$
At this point we recall Lemma \ref{elementary}, and conclude our claim. Thus, there exists a constant $C_0$ that depends only on $p_0^\prime$, $d_1$, $d_2$, $\Omega$ and $M_{13}$ such that 
$$\sup_{\tau \geq 0} \norm{u_1+u_3}_{L^{p_0}(\Omega_{\tau,\tau+1})} \leq C_0.$$
This together with the non-negativity of $\br{u_i}_{i=1,3}$ finishes the proof of Lemma \ref{lem:some_L_p_estimaes_for_123} for $u_1$ and $u_3$.
\end{proof}

\begin{remark}\label{rem:dimension_usage}
It is worthwhile to note that in the above proof the only point where we have used the fact that $N\geq 3$ is in using \eqref{item:less} from Lemma \ref{cor:embedding_for_N} for $p_0^\prime $ which was less than $\frac{N+2}{2}$. In dimensions $1$ and $2$, there is a chance that $p^\prime_0$ will not satisfy this condition, but according to \eqref{item:equal_N} and \eqref{item:greater_N} from the same lemma, we can still get inequality \eqref{eq:some_L_p_estimates_proof_I} {\it for any $q<+\infty$}. As such, the result remains valid when $N=1,2$, as long as condition \eqref{quasi-uniform} is satisfied.
\end{remark}
Now that we have an initial $L^{p_0}\pa{\Omega_{\tau,\tau+1}}$ bound on $u_{i_0}$ and $u_3$, with $i_0=1$ or $i_0=2$, depending on which index satisfies \eqref{quasi-uniform}, we continue to explore the interconnections of the solutions. In particular, we show that any uniform in $\tau$ $L^q\pa{\Omega_{\tau,\tau+1}}$ estimate for $u_3$ can be transferred to $u_1$ and $u_2$. This is, in a sense, a generalisation of \cite[Lemma 33.3]{QS07} in which we exploit the uniform $L^1(\Omega)$ bounds.
\begin{lemma}\label{lem:second_duality}
If $\sup_{\tau\in \mathbb N}\|u_3\|_{L^p(\Omega_{\tau,\tau+1})} \leq C_p$ for some $1<p<\infty$, then
\begin{equation}\nonumber
\sup_{\tau\in \mathbb N}\|u_i\|_{L^p(\Omega_{\tau,\tau+1})} \leq \mathcal{C}_p \quad \text{ for } \quad i=1,2,
\end{equation}
where $\mathcal{C}_p$ is a constant that depends only on $C_p$, $p$, $N$, $\Omega$, $d_1$, $d_3$, $M_{ij}$ and $\br{\norm{u_i}_{L^p(\Omega_{{0,1}})}}_{i=1,2}$.
\end{lemma}
\begin{proof}
The proof follows similar ideas to the proof of Lemma \ref{lem:some_L_p_estimaes_for_123}. Once again, we choose a function $\phi_\tau$, which is a shift by $\tau$ of a $C^\infty$ function such that $\phi(0)=0$, $0\leq \phi \leq 1$, and $\phi \vert_{[1,\infty)}\equiv 1$. Let $\psi_{\theta}$ be the solution to \eqref{heat_equation_bw}, with $d=d_1$ and $\theta\in L^{p^\prime}\pa{\Omega_{\tau,\tau+2}}$, where $p^\prime$ is the H\"older conjugate of $p$. Then, using integration by parts, the fact that $\phi_\tau(\tau)=0$, $\psi_\theta\pa{\tau+2,x}=0$ and the equations for $u_1$ and $u_3$ from \eqref{eq:sys}, we find that
\begin{align*}
&\int_{\tau}^{\tau+2}\int_{\Omega}(\varphi_\tau(t) u_1(x,t))\theta(x,t) dxdt = -\int_{\tau}^{\tau+2}\int_{\Omega}(\varphi_\tau(t) u_1(x,t))(\partial_t\psi_\theta(x,t) + d_1\Delta \psi_\theta(x,t))dxdt\\
&= \int_{\tau}^{\tau+2}\int_{\Omega}\psi_{\theta}(x,t)(\partial_t(\varphi_\tau(x,t) u_1(x,t)) - d_1\Delta(\varphi_\tau(t) u_1(x,t)))dxdt\\
&=\int_{\tau}^{\tau+2}\int_{\Omega}\psi_{\theta}(x,t)(\varphi_\tau^\prime(t) u_1(x,t) +\phi_\tau(t)\pa{ u_3(x,t)u_4(x,t)-u_1(x,t)u_2(x,t)})dxdt\\
&=\int_{\tau}^{\tau+2}\int_{\Omega}\psi_{\theta}(x,t)(\varphi_\tau^\prime(t) u_1(x,t) -\phi_\tau(t)\pa{\partial_t u_3(x,t) - d_3\Delta u_3(x,t)})dxdt\\
&= \int_{\tau}^{\tau+2}\int_{\Omega}\psi_{\theta}(x,t)(\varphi_\tau^\prime(t) \pa{u_1(x,t)+u_3(x,t)} - \partial_t(\varphi_\tau(t) u_3(x,t)) + d_3\Delta (\varphi_\tau(t) u_3(x,t)))dxdt\\
&= \int_{\tau}^{\tau+2}\int_{\Omega}\psi_\theta(x,t) \varphi_\tau^\prime(t) (u_1(x,t) + u_3(x,t))dxdt \\
&\;\;\;\; + \int_{\tau}^{\tau+2}\int_{\Omega}\phi_\tau(t)u_3(x,t)(\underbrace{\partial_t \psi_\theta(x,t) + d_3\Delta \psi_\theta(x,t)}_{\pa{d_3-d_1}\Delta \psi_\theta(x,t) -\theta(x,t)})dxdt\\
&= \int_{\tau}^{\tau+2}\int_{\Omega}\psi_\theta(x,t) \varphi_\tau^\prime(t) (u_1(x,t) + u_3(x,t))dxdt - \int_{\tau}^{\tau+2}\int_{\Omega}\phi_\tau(t)u_3(x,t)\theta(x,t) dxdt\\
&\quad + (d_3 - d_1)\int_{\tau}^{\tau+2}\int_{\Omega}\phi_{\tau}(t)u_3(x,t)\Delta \psi_\theta(x,t) dxdt= \mathcal{I}+\mathcal{II}+\mathcal{III}.
\end{align*}
Using H\"older's inequality, and Lemma \ref{lem_mr_rescaled} we find that
\begin{equation}\label{eq:II_estimation}
\abs{\mathcal{II}}\leq \|\phi_\tau u_3\|_{L^p(\Omega_{\tau,\tau+2})}\|\theta\|_{L^{p'}(\Omega_{\tau,\tau+2})}\leq \| u_3\|_{L^p(\Omega_{\tau,\tau+2})}\|\theta\|_{L^{p'}(\Omega_{\tau,\tau+2})}
\end{equation}
and
\begin{equation}\label{eq:III_estimation}
\begin{split}
\abs{\mathcal{III}} \leq C&\|\phi_\tau u_3\|_{L^p(\Omega_{\tau,\tau+2})}\|\Delta \psi_{\theta}\|_{L^{p'}(\Omega_{\tau,\tau+2})}\\
&\leq C\|u_3\|_{L^p(\Omega_{\tau,\tau+2})}\|\theta\|_{L^{p'}(\Omega_{\tau,\tau+2})},
\end{split}
\end{equation}
where the constants $C$ depends only on $p$, $d_1$, $d_3$, $\Omega$ and $N$. To estimate $\mathcal{I}$ we choose $q < p$ such that its H\"older conjugate, $q^\prime$, satisfies 
$$q^\prime < \begin{cases} \frac{(N+2)p'}{N+2-2p'} &\text{ if } \quad p^\prime < \frac{N+2}{2} \\ p^\prime  &\text{ if } \quad p^\prime \geq \frac{N+2}{2} \end{cases},$$
and with the help of Lemma \ref{cor:embedding_for_N} we conclude that
\begin{equation}\nonumber
\|\psi_{\theta}\|_{L^{q^\prime}(\Omega_{\tau,\tau+2})} \leq C\|\theta\|_{L^{p'}(\Omega_{\tau,\tau+2})}.
\end{equation}
From the above we see that
\begin{align*}
\abs{\mathcal{I}} &\leq C\|\psi_\theta\|_{L^{q'}(\Omega_{\tau,\tau+1})}\pa{\|u_1\|_{L^q(\Omega_{\tau,\tau+1})} + \|u_3\|_{L^q(\Omega_{\tau,\tau+1})}}\\
	&\underset{p>q}{\leq} C\|\theta\|_{L^{p'}(\Omega_{\tau,\tau+2})}\pa{\|u_1\|_{L^q(\Omega_{\tau,\tau+1})} + \|u_3\|_{L^p(\Omega_{\tau,\tau+1})}},
\end{align*}
where we have used the fact that $\phi_\tau^\prime \vert_{[\tau+1,\infty)}=0$. Much like the proof of the previous lemma, since $1<q<p$ we can find $0<\beta<1$ such that the interpolation inequality
$$\|u_1\|_{L^q(\Omega_{\tau,\tau+1})} \leq \|u_1\|_{L^1(\Omega_{\tau,\tau+1})}^{\beta}\|u_1\|_{L^p(\Omega_{\tau,\tau+1})}^{1-\beta} \leq M_{13}^\beta \|u_1\|_{L^p(\Omega_{\tau,\tau+1})}^{1-\beta}$$
is valid. Thus, we find that
\begin{equation}\label{eq:I_estimation}
\abs{\mathcal{I}} \leq C_1\|\theta\|_{L^{p'}(\Omega_{\tau,\tau+2})}\left(\|u_1\|_{L^p(\Omega_{\tau,\tau+1})}^{1-\beta} + \|u_3\|_{L^p(\Omega_{\tau,\tau+2})}\right),
\end{equation}
with $C_1$ depending on $M_{13}$, as well as the other parameters of the problem. Combining \eqref{eq:II_estimation}, \eqref{eq:III_estimation} and \eqref{eq:I_estimation} we obtain
\begin{equation}\nonumber
\int_{\tau}^{\tau+2}\int_{\Omega}(\varphi_\tau(t) u_1(x,t))\theta(x,t) dxdt \leq C_1\|\theta\|_{L^{p'}(\Omega_{\tau,\tau+2})}\pa{\|u_1\|_{L^p(\Omega_{\tau,\tau+1})}^{1-\beta} + \|u_3\|_{L^p(\Omega_{\tau,\tau+2})}},
\end{equation}
which implies by duality that
\begin{equation}\label{eq:norm_of_phi_u_1}
\|\varphi_\tau u_1\|_{L^p(\Omega_{\tau,\tau+2})} \leq C_1\pa{\|u_1\|_{L^p(\Omega_{\tau,\tau+1})}^{1-\beta} + \|u_3\|_{L^p(\Omega_{\tau,\tau+2})}}.
\end{equation}
Again, like in the proof of the previous lemma, it is enough to find a bound for the sequence
$$a_n=\|u_1\|_{L^p(\Omega_{n,n+1})}$$
for all $n\in \N\cup\br{0}$ for which $a_{n+1}\geq a_n$.
Using the fact that $\phi_{\tau}\vert_{[\tau+1,\infty)}\equiv 1$ and the non-negativity of $u_1$ we see that \eqref{eq:norm_of_phi_u_1} implies that 
\begin{equation}\nonumber
a_{n+1} \leq \|\varphi_n u_1\|_{L^p(\Omega_{n,n+2})} \leq  C_1\pa{a_n^{1-\beta} + 2\sup_{\tau\in\N}\|u_3\|_{L^p(\Omega_{\tau,\tau+1})}}.
\end{equation}
If, in addition, $a_{n+1}\geq a_n$ then since for any $0<\beta<1$ and $a,b>0$
$$ab^{1-\beta} \leq \beta a^{\frac{1}{\beta}}+\pa{1-\beta}b$$
we see that
\begin{equation}\nonumber
a_{n+1} \leq C_1 a_{n+1}^{1-\beta} + 2C_1C_p \leq \pa{1-\beta}a_{n+1}+\pa{\beta C_1^{\frac{1}{\beta}}+2C_1C_p}
\end{equation}
from which the independent in $n$ bound
$$a_{n+1}\leq \frac{\beta C_1^{\frac{1}{\beta}}+2C_1C_p}{\beta}$$
arises. The proof is thus complete for $u_1$, and the exact same considerations show the result for $u_2$.
\end{proof}
\begin{remark}\label{rem:not_working_for_infinity}
In general, Lemma \ref{lem:second_duality} is not applicable for $p=\infty$, as our embedding theorems do not always cover this case.
\end{remark}
\begin{lemma}\label{lem:L_infty_bound_on_u_3}
Assume that condition \eqref{quasi-uniform} holds. Then 
\begin{equation}\label{eq:L_infty_bound_on_u_3}
\sup_{t\geq 0}\|u_3(t)\|_{L^\infty(\Omega)}<+\infty.
\end{equation}
Moreover, for any $1\leq p <\infty$, there exists $C_p>0$, that depends only on the parameters of the problem and not $\tau$, such that
\begin{equation}\label{eq:Lp_bounds_on_u_1_u_2_u_3}
\max_{i=1,2,3}\sup_{\tau\geq 0}\|u_i\|_{L^p(\Omega_{\tau,\tau+1})} \leq C_p.
\end{equation}
\end{lemma}
\begin{proof}
Using the time cut off function $\phi_\tau$, defined in the Proof of Lemma \ref{lem:some_L_p_estimaes_for_123}, and the non-negativity of the solutions, we see that
\begin{equation}\label{eq_u3}
\begin{aligned}
\partial_t(\varphi_\tau(t) u_3(x,t)) &- d_3\Delta(\varphi_\tau(t) u_3(x,t)) \\
&=\phi^\prime_{\tau}(t)u_3(x,t)+\phi_\tau(t)\pa{u_1(x,t)u_2(x,t)-u_3(x,t)u_4(x,t)}\\
&\leq \varphi_\tau^\prime(t) u_3(x,t) + \varphi_\tau(t) u_1(x,t)u_2(x,t):=g(x,t).
\end{aligned}
\end{equation}
and
$$\nabla (\varphi_\tau(t) u_3(x,t))\cdot \nu(x)\vert_{x\in \partial \Omega} = 0, \quad \varphi_\tau(\tau) u_3(x,\tau) = 0.$$
This implies that $\phi_\tau(x,t)u_3(x,t)$ is a sub-solution to the heat equation
\begin{equation}\label{eq:sub_sol_heat}
\begin{cases}
\partial_t \psi(x,t) - d_3\Delta \psi(x,t) = g(x,t), &(x,t)\in \Omega_{\tau,\tau+2},\\
\nabla \psi(x,t) \cdot \nu(x) = 0, &(x,t)\in \partial\Omega\times \pa{\tau,\tau+2},\\
\psi(x,\tau) = 0, &x\in\Omega.
\end{cases}
\end{equation}
Lemma \ref{lem:some_L_p_estimaes_for_123} assures us that there exists $C$, depending only on the parameters of the problem and not $\tau$, and $p_0>2$ such that
\begin{equation}\nonumber
\begin{gathered}
\sup_{\tau \geq 0}\|g\|_{L^{\frac{p_0}{2}}(\Omega_{\tau,\tau+2})} \leq C\sup_{\tau \geq 0}\norm{u_3}_{L^{\frac{p_0}{2}}(\Omega_{\tau,\tau+2})}+\sup_{\tau \geq 0}\norm{u_1u_2}_{L^{\frac{p_0}{2}}(\Omega_{\tau,\tau+2})}\\
\leq C\abs{\Omega}^{\frac{1}{p_0}}\sup_{\tau \geq 0}\norm{u_3}_{L^{p_0}(\Omega_{\tau,\tau+2})}+\sup_{\tau \geq 0}\norm{u_1}_{L^{p_0}(\Omega_{\tau,\tau+2})}\norm{u_2}_{L^{p_0}(\Omega_{\tau,\tau+2})}\leq C.
\end{gathered}
\end{equation}
The comparison principle for heat equation assures us that since $\phi_\tau u_3$ is a sub-solution to \eqref{eq:sub_sol_heat},
\begin{equation*}
	\varphi_\tau (t)u_3(x,t) \leq \psi(x,t) \quad \text{for all } \quad (x,t)\in\Omega_{\tau,\tau+2}.
\end{equation*}
From Lemma \ref{cor:embedding_for_N}, we conclude that if $p_0<{N+2}$ then for any $\eta>0$ such that
$$p_1-\eta=\frac{(N+2)\frac{p_0}{2}}{N+2-2\frac{p_0}{2}}-\eta >1$$
we have that $\psi\in L^{p_1-\eta}\pa{\Omega_{\tau,\tau+2}}$, and consequently $\phi_\tau u_3\in L^{p_1-\eta}\pa{\Omega_{\tau,\tau+2}}$. It is important to note that the embedding constant \emph{is independent of $\tau$}. Due to the fact that $\phi_{\tau}\vert_{[\tau+1,\infty)}\equiv 1$ we can conclude that
\begin{equation}\nonumber
\sup_{\tau \geq 0}\norm{u_3}_{L^{p_1-\eta}(\Omega_{{\tau+1,\tau+2}})} < \infty,
\end{equation}
and since $u_3\in L^\infty\pa{\Omega_{T}}$ for any $T>0$, we have that 
$$\sup_{\tau\in \N}\norm{u_3}_{L^{p_1-\eta}(\Omega_{\tau,\tau+1})}=\max\pa{\norm{u_3}_{L^{p_1-\eta}(\Omega_{1})},\sup_{\tau \geq 0}\norm{u_3}_{L^{p_1-\eta}(\Omega_{{\tau+1,\tau+2})}} }<\infty.$$
From Lemma \ref{lem:second_duality} we know that
\begin{equation}
\sup_{\tau\in \N}\norm{u_i}_{L^{p_1-\eta}(\Omega_{\tau,\tau+1})}<\infty,\quad i=1,2
\end{equation}
which in turn implies that $g\in L^{\frac{p_1-\eta}{2}}\pa{\Omega_{\tau,\tau+2}}$ with a uniform in $\tau$ bound. Repeating this procedure, we define $p_{k+1}=\frac{(N+2)\frac{p_k}{2}}{N+2-p_k}$ when $p_k<N+2$, and find for any sufficiently small $\eta$ 
$$\sup_{\tau\in \N}\norm{u_3}_{L^{p_{k+1}-\eta}(\Omega_{\tau,\tau+1})}<\infty.$$
Similarly to the proof of Proposition \ref{global3D}, since $p_0>\frac{N+2}{2}$, there exists $k_0\in \mathbb N$ such that
\begin{equation*}
	p_{k_0} \geq  N+2.
\end{equation*}
If the inequality is strict, then by choosing $\eta>0$ small enough for which $p_{k_0}-\eta>N+2$, we will find that
$$\sup_{\tau\in \N}\norm{u_3}_{L^{p_{k_0}-\eta}(\Omega_{\tau,\tau+1})}<\infty,$$
and the same bound will transfer to $u_1$ and $u_2$, thanks to Lemma \ref{lem:second_duality}. This will imply that $u_1u_2 \in L^{\frac{p_{k_0}-\eta}{2}}(\Omega_{\tau,\tau+2})$ with  $\frac{p_{k_0}-\eta}{2}>\frac{N+2}{2}$, and according to Lemma \ref{cor:embedding_for_N} we find that 
$$\norm{u_3}_{L^\infty(\Omega_{\tau,\tau+1})}<C,$$
for $C$ independent of $\tau$. As $u_3$ is continuous in its variables
$$\sup_{\tau\leq t\leq \tau+1}\norm{u_3(t)}_{L^\infty(\Omega)}\leq \norm{u_3}_{L^\infty(\Omega_{\tau,\tau+1})}$$
and \eqref{eq:L_infty_bound_on_u_3} is shown. The desired estimates \eqref{eq:Lp_bounds_on_u_1_u_2_u_3} for $u_1$ and $u_2$ follow again from Lemma \ref{lem:second_duality}.\\
We are only left to show that we can find $p_{k_0}$ as above. Since $p_{k+1}=\frac{(N+2)\frac{p_k}{2}}{N+2-p_k}$ implies
$$p_{k}=\frac{2(N+2)p_{k+1}}{N+2+2p_{k+1}}$$
we see that the case where $p_{k_0}=N+2$ {\it can come from only one choice of $p_0>\frac{N+2}{2}$.} Since all inequalities will remain the same if we replace $p_0$ with $\widetilde{p_0}< p_0$ such that $\widetilde{p_0} > \frac{N+2}{2}$ (as the domain is bounded), choosing $\widetilde{p_0}$ close enough to $p_0$ such that 
$$\widetilde{p_0} > \frac{2(N+2)p_{0}}{N+2+2p_{0}}$$
assures us the sequence we will construct using $\widetilde{p_0}$ will not pass through $p_0$ and as such the $\widetilde{p_{k_0}}$ we will find from this process will satisfy $\widetilde{p_{k_0}}>N+2$. The proof is thus complete.
\end{proof}
The penultimate component to prove Proposition \ref{Global3D} is the following lemma:
\begin{lemma}\label{3D_u4LN2}
Assume that condition \eqref{quasi-uniform} holds. Then there exists $\alpha \in (0,1)$ such that
	\begin{equation*}
		\norm{u_1(t)}_{L^{\frac N2}(\Omega)} + \norm{u_2(t)}_{L^{\frac N2}(\Omega)} + \norm{u_4(t)}_{L^{\frac N2}(\Omega)} \leq C(1+t)^\alpha \quad \text{ for all } \quad t\geq 0.
	\end{equation*}
\end{lemma}
\begin{proof}
We start by noticing that due to Lemma \ref{lem:L_infty_bound_on_u_3} we can find a constant $C_p$, for any $1<p<+\infty$ such that 
$$\sup_{\tau \geq 0}\norm{u_i}_{L^p\pa{\Omega_{\tau,\tau+1}}}\leq C_p,\quad i=1,2.$$
As such, for any $t>0$
\begin{equation}\label{eq:L_p_estimation_on_omega_t_12}
\norm{u_i}^p_{L^p\pa{\Omega_{t}}}\leq \sum_{n=0}^{\rpa{t}+1}\norm{u_i}^p_{L^p\pa{\Omega_{n,n+1}}} \leq C^p_{p}\pa{2+\rpa{t}}\leq 3C^p_p\pa{1+t},\quad i=1,2.
\end{equation}
Next, using the equation for $u_4$ and the non-negativity of the solutions, we see that 
$$\partial_t u_4(x,t) \leq u_1(x,t)u_2(x,t)$$
from which, together with Young's inequality, we conclude that for any $1<q<+\infty$ and an arbitrary $\epsilon\in(0,1)$ 
	\begin{equation*}
	\begin{gathered}
		\partial_t\int_{\Omega}u_4(x,t)^qdx  \leq q\int_{\Omega}u_4(x,t)^{q-1}u_1(x,t)u_2(x,t)dx \leq \frac{q(q-1)}{q-\epsilon}\int_{\Omega}u_4(x,t)^{q-\varepsilon}dx\\
		 + \frac{q(1-\varepsilon)}{q-\varepsilon}\int_{\Omega}(u_1(x,t)u_2(x,t))^{\frac{q-\varepsilon}{1-\varepsilon}}dx
		\leq \mathcal{C}_{q,\epsilon,\Omega}\left(\int_{\Omega}u_4(x,t)^qdx \right)^{1-\frac{\varepsilon}{q}} \\
		+ \frac{1-\epsilon}{2(q-\epsilon)}\pa{\norm{u_1(t)}_{L^{\frac{2(q-\epsilon)}{1-\epsilon}}\pa{\Omega}}^{\frac{2(q-\epsilon)}{1-\epsilon}}+\norm{u_2(t)}_{L^{\frac{2(q-\epsilon)}{1-\epsilon}}\pa{\Omega}}^{\frac{2(q-\epsilon)}{1-\epsilon}}}
		\end{gathered}
	\end{equation*}
	Applying Lemma \ref{Gronwall} with $y(t) = \norm{u_4(t)}_{L^q(\Omega)}^q$ and $r = \varepsilon/q$, and using \eqref{eq:L_p_estimation_on_omega_t_12}, we find that
	\begin{equation}\label{eq:u4_Lq_bound_time}
	\begin{gathered}
		\norm{u_4(t)}_{L^q(\Omega)}^q \leq C_{q,\epsilon}\pa{\norm{u_{4,0}}_{L^q(\Omega)}^q + \norm{u_1(t)}_{L^{\frac{2(q-\epsilon)}{1-\epsilon}}\pa{\Omega_{t}}}^{\frac{2(q-\epsilon)}{1-\epsilon}}+\norm{u_2(t)}_{L^{\frac{2(q-\epsilon)}{1-\epsilon}}\pa{\Omega_{t}}}^{\frac{2(q-\epsilon)}{1-\epsilon}}+ t^{\frac{q}{\varepsilon}}}\\
		\leq \mathcal{C}_{q,\epsilon}\pa{\norm{u_{4,0}}_{L^q(\Omega)}^q + \pa{1+t}+ t^{\frac{q}{\varepsilon}}}.
		\end{gathered}
	\end{equation}
Thus, we conclude that
	\begin{equation*}
	\begin{gathered}
		\norm{u_4(t)}_{L^q(\Omega)} \leq C_{q,\epsilon}\pa{\norm{u_{4,0}}_{L^q(\Omega)} + (1+t)^\frac{1}{q}+ t^{\frac{1}{\varepsilon}} }\\
		\leq C_{q,\epsilon}\pa{1+\norm{u_{4,0}}_{L^\infty(\Omega)}\abs{\Omega}^{\frac{1}{q}} }\pa{ (1+t)^\frac{1}{q}+ t^{\frac{1}{\varepsilon}} }
		\end{gathered}
	\end{equation*}
for any $1<q<+\infty$ and $\epsilon\in (0,1)$. Since
$$\frac{1}{q}=\frac{\alpha}{4q}+\frac{1-\alpha}{1}$$ 
for $\alpha=\frac{4q-4}{4q-1}$, we have that the interpolation
$$\norm{u(t)}_{L^{q}(\Omega)}\leq \norm{u(t)}_{L^{4q}(\Omega)}^{\frac{4q-4}{4q-1}}\norm{u(t)}_{L^1(\Omega)}^{\frac{3}{4q-1}}$$
is valid for any $q> 1$. As such
$$\norm{u_4(t)}_{L^{q}(\Omega)}\leq C_{q,\epsilon}\pa{M_{2,4}}^{\frac{3}{4q-1}}\pa{1+\norm{u_{4,0}}_{L^\infty(\Omega)}\abs{\Omega}^{\frac{1}{4q}} }^{\frac{4q-4}{4q-2}}\pa{ (1+t)^{\frac{1}{4q}}+ t^{\frac{1}{\varepsilon}} }^{\frac{4q-4}{4q-1}}.$$
For any $\eta<1$ such that
$$\frac{4q-4}{4q-1}<\eta$$
we see that since $\frac{1}{4q}\frac{4q-4}{4q-1} < \frac{1}{4q}$, choosing $\epsilon=\frac{4q-4}{\eta(4q-1)}$ yields the bound 
\begin{equation}\label{eq:bound_u4_typ}
\norm{u_4(t)}_{L^{q}(\Omega)}\leq  C_{q,\alpha,u_{4,0},M_{2,4}}\pa{1+t}^{\alpha},
\end{equation}
with $\alpha=\max\pa{\frac{1}{4q},\eta}<1$. In particular, one can always choose
$$\eta= \frac{4q-1}{4q+2}$$
and in the case $q=\frac{N}{2}$ with $N\geq 3$, this translates into 
$$\alpha = \max\pa{\frac{3}{2N},\frac{2N-1}{2N+2}}.$$
We turn our attention to $u_1$ and $u_2$. We will focus only $u_1$, as it and $u_2$ solve the same type of equation, and as such the desired result, once proven, will also hold for $u_2$.\\ 
We remind ourselves that $u_1$ solves the equation
$$\partial_{t}u_1(x,t)-d_1\Delta u_1(x,t)=u_3(x,t)u_4(x,t)-u_1(x,t)u_2(x,t),$$
and we mimic the steps of Lemma \ref{lem:some_L_p_estimaes_for_123}. We start by choosing a function $\phi_\tau$, which is a shift by $\tau$ of a $C^\infty$ function such that $\phi(0)=0$, $0\leq \phi \leq 1$, and $\phi \vert_{[1,\infty)}\equiv 1$. 
Due to the non-negativity of the solutions to \eqref{eq:sys} we have that
$$\partial_{t}\pa{\phi_\tau(t)u_1(x,t)} - d_1\Delta(\varphi_\tau(t)u_1(x,t) \leq \varphi^\prime_\tau(t) u_1(x,t)+ \phi_{\tau}(t)u_3(x,t)u_4(x,t)\in  L^{q}\pa{\Omega_{\tau,T}}$$
for any $1<q<\infty$ and $0\leq \tau<T$. We can also choose $\phi$ to be increasing, so that the right hand side of the above inequality is a non-negative function. Thus, using the comparison principle for the heat equation, together with the same elements of the proof of Lemma \ref{cor:embedding_for_N} which relied on the maximum regularity principle of the heat equation, and noticing that $\phi_{\tau}u_1 \vert_{t=\tau}=0$ we find that for any $n\in\N$
$$\sup_{n+1\leq t \leq n+2}\norm{u_1(t)}_{L^\infty\pa{\Omega}}\leq \sup_{n\leq t \leq n+2}\norm{\phi_{n}(t)u_1(t)}_{L^\infty\pa{\Omega}}\leq \norm{\phi_n u_1}_{L^{\infty}\pa{\Omega_{n,n+2}}}$$ 
$$\leq C_{q_0,\Omega}\pa{\norm{\phi_n^\prime u_1}_{L^{q_0}\pa{\Omega_{n,n+2}}}+\norm{u_3u_4}_{L^{q_0}\pa{\Omega_{n,n+2}}}}$$
where $q_0>\frac{N+2}{2}$ is fixed. Lemma \ref{lem:L_infty_bound_on_u_3} assures us that 
$$\sup_{n\in\N}\norm{\phi^\prime_n u_1}_{L^{q_0}\pa{\Omega_{n,n+2}}} \leq C\pa{\sup_{n\in\N}\norm{ u_1}_{L^{q_0}\pa{\Omega_{n,n+1}}}+\sup_{n\in\N}\norm{ u_1}_{L^{q_0}\pa{\Omega_{n+1,n+2}}}}\leq C_{q_0},$$
and due to \eqref{eq:bound_u4_typ} we see that
$$\norm{u_3u_4}_{L^{q_0}\pa{\Omega_{n,n+2}}}\leq C_{q_0,\alpha}C_{q_0,\alpha,u_{4,0},M_{2,4}}\sup_{t\geq 0}\norm{u_3(t)}_{L^\infty\pa{\Omega}}\pa{\int_{n}^{n+2} \pa{1+s}^{q_0\alpha} ds}^{\frac{1}{q_0}}.$$
Thus, since $\int_n^{n+2}(1+s)^{q_0\alpha}ds\leq 2(3+n)^{q_0\alpha}$,
$$\sup_{n+1\leq t \leq n+2}\norm{u_1(t)}_{L^\infty\pa{\Omega}}\leq C_{q_0,\Omega}\pa{C_{q_0}+C_{q_0,\alpha,u_{4,0},M_{2,4}}\sup_{t\geq 0}\norm{u_3(t)}_{L^\infty\pa{\Omega}}\pa{3+n}^{\alpha}}$$
$$\leq C_{q_0,\alpha,u_{4,0},M_{2,4},\sup_{t\geq 0}\norm{u_3(t)}_{L^\infty\pa{\Omega}}}\pa{2+t}^{\alpha}.$$
As $n\in\N$ was arbitrary and
$$\sup_{0\leq t \leq 1}\norm{u_1(t)}_{L^\infty\pa{\Omega}} =\norm{u_1}_{L^\infty\pa{\Omega_{1}}}\leq \norm{u_1}_{L^\infty\pa{\Omega_{1}}}\pa{1+t}^\alpha$$
we find that
\begin{equation}\label{eq:u1_infty_bound}
\norm{u_1(t)}_{L^\infty\pa{\Omega}} \leq C_{q_0,\alpha,\br{u_{i,0}}_{i=1,\dots,4},M_{2,4}}\pa{1+t}^\alpha,
\end{equation} 
where the constant depends on $\norm{u_1}_{L^\infty\pa{\Omega_{1}}}$, which by itself depends on the initial data. From the above any $L^q\pa{\Omega}$ bound follows. As $\alpha<1$, the proof is completed.
\end{proof}

\begin{remark}\label{rem:N_independent}
It is worth to mention that while Lemma \ref{3D_u4LN2} considered $L^{\frac{N}{2}}\pa{\Omega}$, we have actually managed to prove that for any $q>1$, we can find $0<\alpha<1$ such that 
\begin{equation}\nonumber
\norm{u_4(t)}_{L^{q}(\Omega)}\leq  C_{q,\alpha,u_{4,0},M_{2,4}}\pa{1+t}^{\alpha},
\end{equation}
and
\begin{equation}\nonumber
\norm{u_i(t)}_{L^\infty\pa{\Omega}} \leq C_{\alpha,\br{u_{i,0}}_{i=1,\dots,4},M_{2,4}}\pa{1+t}^\alpha,\quad i=1,2.
\end{equation}
Moreover, the dimension $N$ played no role! 
\end{remark}
Collecting all the tools that we have developed, we can finally prove Proposition \ref{Global3D}.
\begin{proof}[Proof of Proposition \ref{Global3D} ]
This follows from Proposition \ref{global3D} (where \eqref{3D_u124_Linfty} has also been proved), Lemma \ref{lem:L_infty_bound_on_u_3} and Lemma \ref{3D_u4LN2}.
\end{proof}

\subsection{Proof of Proposition \ref{Global1D}}\label{12D_proof}
Looking at the main ingredients we used to prove  Proposition \ref{Global3D} (Proposition \ref{global3D}, Lemma \ref{lem:L_infty_bound_on_u_3} and Lemma \ref{3D_u4LN2}), we notice that there is no real dimension dependency in the proof any of them. The only dimensional dependency may come from condition \eqref{quasi-uniform}. We also note that the only difference between the conclusions of Proposition \ref{Global1D} and \ref{Global3D} is the $L^p\pa{\Omega}$ norm estimates - namely \eqref{eq:bound_u_1_u_2} and \eqref{3D_u124}. However, it is simple to see that condition \eqref{eq:bound_u_1_u_2} in Proposition \ref{Global1D} is an immediate consequence of \eqref{eq:bound_u124} and the mass conservation. Indeed, for $i\in \{1,\ldots, 4\}$,
$$\norm{u_i(t)}_{L^{1+\gamma}\pa{\Omega}} \leq \norm{u_i(t)}_{L^\infty\pa{\Omega}}^{\frac{\gamma}{1+\gamma}}\norm{u_i(t)}^{\frac{1}{1+\gamma}}_{L^{1}\pa{\Omega}}$$
$$\leq C^{\frac{\gamma}{1+\gamma}}\pa{M_{1,3}+M_{2,4}}^{\frac{1}{1+\gamma}}\pa{1+t}^{\frac{\gamma\mu}{1+\gamma}}.$$
Thus, choosing $\gamma=\frac{\epsilon}{\mu-\epsilon}$ gives the desired estimate. As \eqref{eq:bound_u124} follows from Proposition \ref{global3D}, we only truly need it and Lemma \ref{lem:L_infty_bound_on_u_3} to conclude Proposition \ref{Global1D}.

From the above discussion we conclude that in order to prove Proposition \ref{Global1D} it is enough to show that when $N=1,2$, condition \eqref{quasi-uniform} is \emph{always satisfied}.
\begin{lemma}\label{improved_duality_12D}
	For any $\omega_1, \omega_2>0$, there exists a constant $p_*<2$ such that
	\begin{equation}\label{b4}
		\frac{\abs{\omega_1-\omega_2}}{\omega_1+\omega_2}C_{\rm{mr},p_*}<1.
	\end{equation}
In particular, condition \eqref{quasi-uniform} is always satisfied.
\end{lemma}
\begin{proof}
	First, we show that
	\begin{equation}\label{Cmr2}
		C_{\rm{mr},2} \leq 1.
	\end{equation}
	Indeed, by multiplying \eqref{normalized_diffusion} with $\Delta \psi$ and integrating on $\Omega_{\tau,T}$, we find that 
	\begin{equation*}
\frac{1}{2}\int_{\Omega}\abs{\nabla \psi(x,\tau)}^2dx+ \norm{\Delta \psi}_{L^2(\Omega_{\tau,T})}^2 = -\int_\tau^T\int_{\Omega}\theta(x,t) \Delta \psi(x,t) dxdt$$
$$ \leq \frac 12 \norm{\Delta \psi}_{L^2(\Omega_{\tau,T})}^2 + \frac 12 \norm{\theta}_{L^2(\Omega_{\tau,T})}
	\end{equation*}
where we have used the fact that
$$\int_{\Omega_{\tau,T}}\partial_t\psi(x,t) \Delta \psi(x,t)dxdt \underset{\text{Neumann condition}}{=}-\int_{\Omega_{\tau,T}}\nabla\pa{\partial_t\psi(x,t)} \nabla\psi(x,t)dxdt$$
$$=-\frac{1}{2}\int_{\Omega_{\tau,T}}\partial_t\abs{\nabla \psi(x,t)}^2dtdx\underset{\psi(x,t)=0}{=}\int_{\Omega}\abs{\nabla \psi(x,\tau)}^2dx.$$
Thus,
	\begin{equation*}
		\norm{\Delta \psi}_{L^2(\Omega_{\tau,T})} \leq \norm{\theta}_{L^2(\Omega_{\tau,T})},
	\end{equation*}
	which proves \eqref{Cmr2}. It is then obvious that for $\omega_1,\omega_2>0$
	\begin{equation*}
		\frac{|\omega_1 - \omega_2|}{\omega_1 +\omega_2}C_{\rm{mr}, 2} < 1.
	\end{equation*}
To conclude our lemma, it will be sufficient to show that 
	\begin{equation*}
		C_{\rm{mr},2}^- := \liminf_{\eta \to 0^+}C_{\rm{mr},2-\eta} \leq C_{\rm{mr},2}.
	\end{equation*}
To see that, we follow the idea in \cite[Remark 4]{PSU17}. Let $2_\eta$ satisfy
	\begin{equation*}
		\frac{1}{2_\eta} = \frac 12 \rpa{\frac 12  + \frac{1}{2-\eta}} \quad \text{ or equivalently } \quad 2_\eta = 2 - \frac{2\eta}{4-\eta}.
	\end{equation*}
Applying the Riesz-Thorin interpolation theorem (see, for instance, \cite[Chapter 2]{Lun18}), we find that
	\begin{equation*}
		C_{\rm{mr},2_\eta} \leq C_{\rm{mr},2}^{\frac 12}C_{\rm{mr},2-\eta}^{\frac 12}.
	\end{equation*}
	This implies that
	\begin{equation*}
		C_{\rm{mr},2}^{-} \leq C_{\rm{mr},2}^{\frac 12}\pa{C_{\rm{mr},2}^{-}}^{\frac 12}
	\end{equation*}
	and therefore
	\begin{equation*}
		C_{\rm{mr},2}^{-} \leq C_{\rm{mr},2}.
	\end{equation*}
We are only left with showing that the above implies that condition \eqref{quasi-uniform} is always satisfied. Given any diffusion coefficients, let $p_0^\prime$ be the exponent $p_\ast<2$ that we found, which can depend on the coefficients. As
$$p_0=\frac{p_0^\prime}{p_0^\prime-1}=\frac{p_\ast}{p_\ast-1}=1+\frac{1}{p_\ast-1}>2$$
and $\frac{N+2}{2} \leq 2$ when $N=1,2$, the proof is complete. 
\end{proof}

\section{The interaction between the entropy and the norm bounds: Proof of the main theorems}\label{sec:proof}
Until this point we have considered two aspects of our system of equations, \eqref{eq:sys}:
\begin{itemize}
\item The entropy of the system, and the functional inequality that governs its dissipation under the flow of our system.
\item The well-posedness of the solution to the equation, and time dependent estimates on the growth of norms of interest. 
\end{itemize}
Using the second point in the above will allow us to show convergence to equilibrium in the entropic sense. However, the entropy itself controls $L^1(\Omega)$-norm, a fact that will allow us to \emph{revisit and improve} our norm estimates, which in turn will give better entropic convergence and so on and so forth. This idea of bootstraping ourselves, not only using the equations themselves, but also their interplay with the entropy, is the key to prove our main theorems.

We start with this ``entropic control'' - which is known as a Csisz\'ar-Kullback-Pinsker-type inequality, see \cite[Lemma 2.3]{FT17}.
\begin{lemma}\label{CKP}
	There exists a constant $C_{CKP}>0$ such that, for any measurable non-negative functions $\bm{u} = (u_i)_{i=1,\ldots, 4}$ who satisfy
	\begin{equation*}
		\int_{\Omega}(u_i(x)+u_{j}(x))dx = |\Omega|(u_{i,\infty}+u_{j,\infty}) \quad \text{ for } \quad i\in \{1,2\}, \; j\in \{3,4\},
	\end{equation*}
	we have
	\begin{equation*}
		H(\bm{u}|\bm{u}_\infty) \geq C_{CKP}\sum_{i=1}^4\norm{u_i - u_{i,\infty}}_{L^1(\Omega)}^2.
	\end{equation*}
	
\end{lemma}
The Csisz\'ar-Kullback-Pinsker inequality will allow us to pass the stretched exponential convergence of the entropy, which is guaranteed due to our estimates, to a stretched exponential convergence to equilibrium for \eqref{eq:sys} in $L^1(\Omega)$-norm, and consequently {\it in any} $L^p(\Omega)$ norm.
\begin{proposition}\label{stretched_exp}
Consider the system of equations \eqref{eq:sys}, where $\Omega\in \R^N$, with $N\geq 3$, is a bounded domain with smooth boundary. Assume in addition that condition \eqref{quasi-uniform_theorem} is satisfied. Then for any $1<p<\infty$, there exist $C_{p,\br{u_{i,0}}_{i=1,\dots,4}},c_p>0$ and $\varepsilon_p>0$ such that
\begin{equation}\label{eq:stretched_Lp}
	\sum_{i=1}^4\norm{u_i(t) - u_{i,\infty}}_{L^p(\Omega)} \leq C_{p,\br{u_{i,0}}_{i=1,\dots,4}} e^{-c_p(1+t)^{\varepsilon_p}} \quad \text{ for all } \quad t\geq 0.
\end{equation}
The above remains valid when $N=1,2$, without any conditions on the diffusion coefficients.
\end{proposition}
\begin{proof}
When $N\geq 3$, using Proposition \ref{Global3D}, we conclude that the estimates \eqref{3D_u3}, \eqref{3D_u124_Linfty} and \eqref{3D_u124} are valid for some $0<\alpha<1$ and $\mu>0$. Invoking Theorem \ref{thm:convergence_general}, we find that
\begin{equation*}
	H(\bm{u}(t)|\bm{u}_\infty) \leq H(\bm{u}_0|\bm{u}_\infty)e^{-C(\alpha,\varepsilon)(1+t)^{1-\alpha-\varepsilon}}
\end{equation*}
for any $0<\varepsilon<1-\alpha$ and some $C(\alpha,\varepsilon)>0$. Applying the Csisz\'ar-Kullback-Pinsker inequality we have that
\begin{equation}\label{eq:L1_entropy_bound}
\sum_{i=1}^4\norm{u_i(t) - u_{i,\infty}}_{L^1(\Omega)} \leq C_{CKP}^{-1}H(\bm{u}_0|\bm{u}_\infty)e^{-C(\alpha,\varepsilon)(1+t)^{1-\alpha-\varepsilon}} \quad \text{ for all } \quad t\geq 0.
\end{equation}
Using the fact that 
$$\norm{u}_{L^p\pa{\Omega}} \leq \norm{u}_{L^\infty\pa{\Omega}}^{1-\frac{1}{p}}\norm{u}^{\frac{1}{p}}_{L^1\pa{\Omega}}$$
together with \eqref{3D_u3}, \eqref{3D_u124_Linfty} and \eqref{eq:L1_entropy_bound} we have that for any $1<p<+\infty$
\begin{align*}
	\sum_{i=1}^4\norm{u_i(t) - u_{i,\infty}}_{L^p(\Omega)} 
	\leq &C_{p,\br{u_{i,0}}_{i=1,\dots,4}}(1+t)^{\mu\left(1 - \frac 1p \right)}e^{-\frac{C(\alpha,\varepsilon)}{p}(1+t)^{1-\alpha-\varepsilon}}\\
	&\leq \mathcal{C}_{p,\br{u_{i,0}}_{i=1,\dots,4},\mu,\alpha,\epsilon}e^{-\frac{C(\alpha,\varepsilon)}{p}(1+t)^{\epsilon_p}},
\end{align*}
for any fixed $0<\varepsilon_p < 1-\alpha-\varepsilon$. This concludes the proof when $N\geq 3$. The cases $N=1,2$ follow from Proposition \ref{Global1D}, which has no conditions on the diffusion coefficients, and the exact same method.
\end{proof}
The above proposition, and the structure of \eqref{eq:sys}, will allow us to boost ourselves even further and obtain a stretched exponential convergence to equilibrium in $L^\infty(\Omega)$ norm for $u_1, u_2$ and $u_3$.

A key observation to achieve this is that
\begin{equation}\label{eq:f_estimation}
\begin{aligned}
f(\bm{u})&=u_1u_2-u_3u_4\\
&=\pa{u_1-u_{1,\infty}}u_2 + u_{1,\infty}\pa{u_2-u_{2,\infty}}
-\pa{u_3-u_{3,\infty}}u_4 - u_{3,\infty}\pa{u_4-u_{4,\infty}},
\end{aligned}
\end{equation}
where we used the fact that $u_{1,\infty}u_{2,\infty}=u_{3,\infty}u_{4,\infty}$.
\begin{proposition}\label{stretched_exp_1}
Under the same conditions of Proposition \ref{stretched_exp} there exist $C,c>0$ and $\varepsilon_0>0$ such that
\begin{equation}\label{eq:stretched_exp_1}
\norm{u_i(t) - u_{i,\infty}}_{L^\infty(\Omega)} \leq Ce^{-c(1+t)^{\varepsilon_0}} \quad \text{ for } \quad i=1,2,3.
\end{equation}
\end{proposition}
\begin{proof}
	For $i=1,2,3$, we denote by $S_i(t) = e^{d_i\Delta t}$ the heat semigroup generated by the operator $-d_i\Delta$ with homogeneous Neumann boundary condition. We will use the classical estimate
	\begin{equation}\label{semigroup}
		\norm{S_i(t)u_0}_{L^\infty(\Omega)} \leq C_{i}t^{-\frac{N}{2p}}\norm{u_0}_{L^p(\Omega)} \quad \text{ for all } \quad t>0,
	\end{equation}
	which can be found in e.g. \cite[Eq. (1.15), page 274]{Tay13}. Denoting by $f_1=-f$, defined in \eqref{eq:f_estimation}, we see that
	\begin{equation*}
		\partial_t(u_1 - u_{1,\infty})(x,t) - d_i\Delta(u_1 - u_{1,\infty})(x,t) = f_1(\bm{u}(x,t))
	\end{equation*}
and as such, using Duhamel's formula on $[t,t+1]$, we see that
	\begin{equation*}
	\begin{aligned}
		u_1(t+1,x) - u_{1,\infty} &= S_i(1)(u_1(x,t) - u_{1,\infty}) + \int_t^{t+1}S_i\pa{t+1-s}f_1(\bm{u}(s,x)ds\\
		&=S_i(1)(u_1(x,t) - u_{1,\infty}) + \int_0^{1}S_i\pa{1-s}f_1(\bm{u}(t+s,x)ds.
		\end{aligned}
	\end{equation*}
Using \eqref{eq:stretched_Lp} and \eqref{semigroup} we find that
	\begin{equation}\label{b6}
	\begin{aligned}
		&\norm{u_1(t+1) - u_{1,\infty}}_{L^\infty(\Omega)}\\
		&\leq C_{i}\norm{u_1(t) - u_{1,\infty}}_{L^p(\Omega)} + \int_0^1\norm{S_i\pa{1-s}f_1(\bm{u}(t+s))}_{L^\infty(\Omega)}ds\\
		&\leq C_{i,p}e^{-c_p(1+t)^{\varepsilon_p}} + C_{i}\int_0^1(1-s)^{-\frac{N}{2p}}\norm{f_1(\bm{u}(t+s))}_{L^p(\Omega)}ds.
	\end{aligned}
	\end{equation}
Equality \eqref{eq:f_estimation}, together with \eqref{eq:stretched_Lp} and the $L^\infty$ bounds on $\br{u_i}_{i=1,\dots,4}$, show that
	\begin{align*}
		&\norm{f_1(\bm{u}(t+s)) }_{L^p(\Omega)}^p\leq \norm{u_2(t+s)}_{L^\infty(\Omega)}^p\norm{u_1(t+s) - u_{1,\infty}}_{L^p(\Omega)}^p + u_{1,\infty}^p\norm{u_2(t+s) - u_{2,\infty}}_{L^p(\Omega)}^p\\
		& + \norm{u_4(t+s)}_{L^\infty(\Omega)}^p\norm{u_3(t+s) - u_{3,\infty}}_{L^p(\Omega)}^p + u_{3,\infty}^p\norm{u_4(t+s) - u_{4,\infty}}_{L^p(\Omega)}^p\\
		& \quad \leq C_{p}(1+t+s)^{\mu p}e^{-c_p(1+t+s)^{\varepsilon_p}}
	\end{align*}
	 Inserting this into \eqref{b6} yields
	\begin{align*}		
		&\norm{u_1(t+1) - u_{1,\infty}}_{L^\infty(\Omega)} \leq C_{i,p}e^{-c_p(1+t)^{\varepsilon_p}} + C_p\int_0^1(1-s)^{-\frac{N}{2p}}(1+t+s)^{\mu}e^{-c_p(1+t+s)^{\varepsilon_p}}ds\\
		&\leq Ce^{-C(1+t)^{\varepsilon_p}}\rpa{1+(2+t)^{\mu}\int_0^1(1-s)^{-\frac{N}{2p}}ds}.
	\end{align*}
As for any $p > N/2$ we have that $\int_0^1(1-s)^{-\frac{N}{2p}}ds < +\infty$ we conclude that choosing $p_0>\frac{N}{2}$ yields
	\begin{equation*}
		\norm{u_1(t+1) - u_{1,\infty}}_{L^\infty(\Omega)} \leq C_{p_0,\br{u_i}_{i=1,\dots,4}}e^{-c_{p_0}(1+t)^{\varepsilon_{p_0}}}\pa{1+(2+t)^{\mu}} \leq C_{p_0,\br{u_i}_{i=1,\dots,4}}e^{-c_{p_0}(1+t)^{\varepsilon_0}}
	\end{equation*}
	for any $0<\varepsilon_0 < \varepsilon_{p_0}$. The same argument is valid for $u_2$ and $u_3$ (here we use $f$ and not $f_1$), which concludes the proof of the proposition.
\end{proof}
\begin{corollary}\label{cor:extra_bondedness}
Under the assumption of Proposition \ref{stretched_exp} we have that 
$$\sup_{t\geq 0}\norm{u_i}_{L^\infty\pa{\Omega}} < +\infty,\quad i=1,2,3.$$
\end{corollary}
\begin{proof}
This follows immediately from \eqref{eq:stretched_exp_1} (though we already knew this for $u_3$).
\end{proof}
Before proving our main theorems we turn our attention to the only missing function: $u_4$.
\begin{proposition}\label{Linf_u4}
Under the assumption of Proposition \ref{stretched_exp} we have that 
\begin{equation*}
\sup_{t\geq 0}\norm{u_4(t)}_{L^\infty(\Omega)} <+\infty.
\end{equation*}
\end{proposition}
\begin{proof}
	From Proposition \ref{stretched_exp_1} we have 
	\begin{equation*}
		\norm{u_3(t) - u_{3,\infty}}_{L^\infty(\Omega)} \leq Ce^{-ct^{\varepsilon_0}}.
	\end{equation*}
	for some fixed constants $C,\epsilon_0>0$. We can find $T_1>0$ large enough such that 
	$$Ce^{-ct^{\varepsilon_0}} \leq \frac 12u_{3,\infty}$$ 
for all $t\geq T_1$, for instance
$$T_1 = \rpa{\frac 1c\ln(2C/u_{3,\infty})}^{1/\varepsilon_0}.$$ 
Clearly, for all $t\geq T_1$ we have that $u_3(x,t) \geq \frac 12u_{3,\infty}$, and as such 
	\begin{equation*}
		\partial_t u_4(x,t) + \frac 12u_{3,\infty}u_4(x,t) \leq \partial_t u_4(x,t) + u_3(x,t)u_4(x,t) = u_1(x,t)u_2(x,t).
	\end{equation*}
Gronwall's lemma, with $\lambda:= \frac 12 u_{3,\infty}$, assures us that 
	\begin{align*}
		u_4(x,t) &\leq e^{-\lambda (t-T_1)}u_4(T_1,x) + \int_{T_1}^te^{-\lambda(t-s)}u_1(s,x)u_2(s,x)ds\\
		&\leq e^{-\lambda (t-T_1)}u_4(T_1,x) + \pa{\sup_{t\geq 0}\norm{u_1(t)}_{L^\infty\pa{\Omega}}}\pa{\sup_{t\geq 0}\norm{u_2(t)}_{L^\infty\pa{\Omega}}}\int_{T_1}^te^{-\lambda (t-s)}ds\\
		&\leq e^{-\lambda (t-T_1)}u_4(T_1,x) + \frac{\pa{\sup_{t\geq 0}\norm{u_1(t)}_{L^\infty\pa{\Omega}}}\pa{\sup_{t\geq 0}\norm{u_2(t)}_{L^\infty\pa{\Omega}}}}{\lambda}
	\end{align*}
The non-negativity of $u_4$ implies that 
$$\sup_{t\geq 0}\norm{u_4(t)}_{L^\infty\pa{\Omega}} \leq \sup_{0\leq t\leq T_1}\norm{u_4(t)}_{L^\infty\pa{\Omega}}+\frac{\pa{\sup_{t\geq 0}\norm{u_1(t)}_{L^\infty\pa{\Omega}}}\pa{\sup_{t\geq 0}\norm{u_2(t)}_{L^\infty\pa{\Omega}}}}{\lambda}$$
and since 
$$\sup_{0\leq t\leq T_1}\norm{u_4(t)}_{L^\infty\pa{\Omega}}<C_{u_{4,0}}\pa{1+T_1}^\mu.$$
The proof is complete.
\end{proof}
With these tools at hand, we are now ready to prove Theorem \ref{thm:main} and \ref{thm:main-3D}.
\begin{proof}[Proof of Theorems \ref{thm:main} and \ref{thm:main-3D}]
The existence of a classical global solution and its non-negativity are covered in Propositions \ref{Global1D} and \ref{Global3D}. We now turn our attention to the $L^\infty\pa{\Omega}$ convergence.\\
	Corollary \ref{cor:extra_bondedness} and Proposition \ref{Linf_u4}, together with Theorem \ref{thm:convergence_general} yield the exponential entropic convergence 
	\begin{equation*}
		H(\bm{u}(t)|\bm{u}_\infty) \leq Ce^{-\lambda_0t}
	\end{equation*}
	for some $C, \lambda_0>0$. Using the Csisz\'ar-Kullback-Pinsker inequality \eqref{CKP} again leads to 
	\begin{equation}\label{exp_L1}
		\sum_{i=1}^4\norm{u_i(t) - u_{i,\infty}}_{L^1(\Omega)}^2 \leq Ce^{-\lambda_0t},
	\end{equation}
	and consequently, for any $1<p<\infty$ 
	\begin{equation}\label{exp_Lp}
		\norm{u_i(t)-u_{i,\infty}}_{L^p(\Omega)} \leq \norm{u_i(t)-u_{i,\infty}}_{L^{\infty}(\Omega)}^{1-\frac 1p}\norm{u_i(t)-u_{i,\infty}}_{L^1(\Omega)}^{\frac 1p} \leq Ce^{-\frac{\lambda_0}{2p}t}
	\end{equation}
	thanks to the uniform in time bounds on $L^\infty\pa{\Omega}$ norms of all functions involved. To finally obtain the exponential convergence in $L^\infty(\Omega)$ norm for $\br{u_i}_{i=1,2,3}$ we use similar arguments to Proposition \ref{stretched_exp_1}.
The function $f$, defined in \eqref{eq:f_estimation}, now has the bound
$$\norm{f(t)}_{L^p\pa{\Omega}}\leq \mathcal{C}e^{-\frac{\lambda_0}{2p}t}$$
and as such \eqref{b6} can be replaced with
	\begin{align*}
	&\norm{u_i(t+1) - u_{i,\infty}}_{L^\infty(\Omega)}
	\leq Ce^{-\frac{\lambda_0}{2p}t}\rpa{1+\int_0^1(1-s)^{-\frac{N}{2p}}ds}\leq Ce^{-\frac{\lambda_0}{2p}t},\quad i=1,2,3
	\end{align*}
if we choose $p>\frac N2$.
We are only left to deal with $u_4$. With $f$ as in \eqref{eq:f_estimation} we see that 
$$\partial_t\pa{u_4-u_{4,\infty}}(x,t) = f(\bm{u}(x,t)),$$
which we can rewrite as
$$\partial_t\pa{u_4-u_{4,\infty}}(x,t)+ u_{3,\infty}\pa{u_4(x,t)-u_{4,\infty}} = \pa{u_1(x,t)-u_{1,\infty}}u_2(x,t)$$
$$ + u_{1,\infty}\pa{u_2(x,t)-u_{2,\infty}}
-\pa{u_3(x,t)-u_{3,\infty}}u_4(x,t) $$
Integrating the above we find that
$$u_4(x,t)-u_{4,\infty} = e^{-u_{3,\infty}t}\pa{u_{4,0}(x)-u_{4,\infty}}$$
$$+\int_{0}^t e^{-u_{3,\infty}\pa{t-s}}\pa{\pa{u_1(s,x)-u_{1,\infty}}u_2(s,x)+ u_{1,\infty}\pa{u_2(s,x)-u_{2,\infty}}
-\pa{u_3(s,x)-u_{3,\infty}}u_4(s,x)}ds$$
which implies that 
$$\norm{u_4(x,t)-u_{4,\infty}}_{L^\infty\pa{\Omega}} \leq e^{-u_{3,\infty}t}\norm{u_{4,0}(x)-u_{4,\infty}}_{L^\infty\pa{\Omega}}+Ce^{-u_{3,\infty}t}\int_{0}^t e^{\pa{u_{3,\infty}-\lambda}s}ds$$
where $\lambda>0$ is such that
$$\norm{u_i(t)-u_{i,\infty}}_{L^\infty\pa{\Omega}} \leq \mathcal{C}e^{-\lambda t},\quad i=1,2,3$$ 
and where we have used the uniform in time $L^\infty\pa{\Omega}$ bounds of $\br{u_i}_{i=1,\dots,4}$. Thus
$$\norm{u_4(x,t)-u_{4,\infty}}_{L^\infty\pa{\Omega}}\leq e^{-u_{3,\infty}t}\norm{u_{4,0}(x)-u_{4,\infty}}_{L^\infty\pa{\Omega}}+\begin{cases}\frac{e^{-\lambda t}-e^{-u_{3,\infty}t}}{u_{3,\infty}-\lambda} & \lambda\not=u_{3,\infty} \\
te^{-u_{3,\infty}t} & \lambda=u_{3,\infty}\end{cases}$$
which shows that for any $\lambda_1 < \min\pa{\lambda, u_{3,\infty}}$, we have that 
$$\norm{u_4(x,t)-u_{4,\infty}}_{L^\infty\pa{\Omega}}\leq C_{\br{u_{i,0}}_{i=1,\dots,3},\lambda_1}e^{-\lambda_1 t}.$$
This concludes the proof.
\end{proof}

\section{Final Remarks}\label{sec:remarks}
We feel that our work is just the beginning of further investigation of systems of the form \eqref{eq:sys_full} and \eqref{eq:sys}, especially in light of the indirect diffusion effect we've shown, and the important interplay between the entropy of the system and properties of the solutions to it.

\medskip
There are many additional interesting open problems, which we hope to address in the near future. A few examples are:
\begin{itemize}
\item \textit{Is it possible to obtain global bounded solutions in dimensions higher than $2$ without assuming additional assumption on the diffusion coefficients?} A positive answer to this question will potentially extend Theorem \ref{thm:main-3D}, and perhaps allow us to consider similar techniques to deal with the more complicated systems \eqref{eq:general_chemical_rd}.
\item \textit{How does the indirect diffusion effect work for more general systems?} For example in systems of the form \eqref{eq:general_chemical_rd} with one or more species without diffusion. We believe that our approach in this paper can be reused once enough estimates on the solutions are available.
\item \textit{Can one obtain an indirect diffusion effect of the form \eqref{IDE} in our setting, where the constant $\beta$ only depends on quantities that are conserved?} While it is interesting, and enlightening, to see how the interaction between the entropic trend to equilibrium and bounds on the solutions lead to showing that \eqref{IDE} is valid for a constant $\beta$ that depends on the $L^\infty\pa{\Omega}-$norm, finding a direct method to show  \eqref{IDE} with the aforementioned $\beta$ gives hope to extending our results to more complicated systems. 
\end{itemize}

\medskip
\par{\bf Acknowledgements:} The third author is supported by the International Research Training Group IGDK 1754 "Optimization and Numerical Analysis for Partial Differential Equations with Nonsmooth Structures", funded by the German Research Council (DFG)  project number 188264188/GRK1754 and the Austrian Science Fund (FWF) under
	grant number W 1244-N18. This work is partially supported by NAWI Graz.
	
The first and third authors gratefully acknowledge the support of the Hausdorff Research Institute for Mathematics (Bonn), through the Junior Trimester Program on Kinetic Theory.

\end{document}